\documentclass[a4paper,11pt]{article}
\usepackage[latin1]{inputenc}
\usepackage[english]{babel}
\usepackage{amsmath}
\usepackage{amsfonts}
\usepackage{amssymb}
\usepackage{epsfig}
\usepackage{amsopn}
\usepackage{amsthm}
\usepackage{color}
\usepackage{graphicx}
\usepackage{subfigure}
\usepackage{enumerate}
\newtheorem{theorem}{Theorem}[section]

\newtheorem{lemma}[theorem]{Lemma}

\newtheorem*{theorem*}{Theorem}
\newtheorem*{lemma*}{Lemma}
\newtheorem*{remark*}{Remark}
\newtheorem*{definition*}{Definition}
\newtheorem*{proposition*}{Proposition}
\newtheorem*{corollary*}{Corollary}
\numberwithin{equation}{section}
\newcommand{\real}{\mathbb{R}}

\let\ced=\c         

\def\qed{\,\unskip\kern 6pt \penalty 500
\raise -2pt\hbox{\vrule \vbox to8pt{\hrule width 6pt
\vfill\hrule}\vrule}\par}
\definecolor{darkblue}{rgb}{0.05, .05, .65}
\definecolor{darkgreen}{rgb}{0.1, .65, .1}
\definecolor{darkred}{rgb}{0.8,0,0}
\newcommand{\beqn}{\begin{equation}}
\newcommand{\eeqn}{\end{equation}}
\newcommand{\bear}{\begin{eqnarray}}
\newcommand{\eear}{\end{eqnarray}}
\newcommand{\bean}{\begin{eqnarray*}}
\newcommand{\eean}{\end{eqnarray*}}
\begin{document}

\title{\huge \bf Self-similar blow-up solutions for the supercritical parabolic Hardy-H\'enon equation}

\author{
\Large Razvan Gabriel Iagar\,\footnote{Departamento de Matem\'{a}tica
Aplicada, Ciencia e Ingenieria de los Materiales y Tecnologia
Electr\'onica, Universidad Rey Juan Carlos, M\'{o}stoles,
28933, Madrid, Spain, \textit{e-mail:} razvan.iagar@urjc.es}, \ \Large Ana Isabel Mu\~{n}oz\,\footnote{Departamento de Matem\'{a}tica
Aplicada, Ciencia e Ingenieria de los Materiales y Tecnologia
Electr\'onica, Universidad Rey Juan Carlos, M\'{o}stoles,
28933, Madrid, Spain, \textit{e-mail:} anaisabel.munoz@urjc.es},
\\[4pt] \Large Ariel S\'{a}nchez\footnote{Departamento de Matem\'{a}tica
Aplicada, Ciencia e Ingenieria de los Materiales y Tecnologia
Electr\'onica, Universidad Rey Juan Carlos, M\'{o}stoles,
28933, Madrid, Spain, \textit{e-mail:} ariel.sanchez@urjc.es}\\
[4pt] }
\date{}
\maketitle

\begin{abstract}
We classify the self-similar solutions presenting finite time blow-up to the parabolic Hardy-H\'enon equation
$$
\partial_tu=\Delta u+|x|^{\sigma}u^p, \quad (x,t)\in\real^N\times(0,\infty),
$$
in dimension $N\geq3$ and the range of exponents
$$
\sigma\in(-2,\infty), \quad p>p_S(\sigma):=\frac{N+2\sigma+2}{N-2}.
$$
We establish the \emph{existence of self-similar blow-up solutions for any $p>p_S(\sigma)$}, provided $\sigma\geq2$. Moreover, we prove that, if $k$ is any natural number and $\sigma\geq 4k-2$, the parabolic Hardy-H\'enon equation has at least $k$ different self-similar blow-up solutions for any $p>p_S(\sigma)$. These results are in a stark contrast with the standard reaction-diffusion equation
$$
\partial_tu=\Delta u+u^p, \quad (x,t)\in\real^N\times(0,\infty),
$$
for which non-existence of any self-similar solution has been established, provided $p$ overpasses the Lepin exponent
$p_L:=1+\frac{6}{N-10}$, $N\geq11$.

For $\sigma\in(-2,2)$, we derive the expression of generalized Lepin exponents $p_L(\sigma)$ for $\sigma\in(0,2)$, respectively $\overline{p_L}(\sigma)$ for $\sigma\in(-2,0)$, and prove existence of self-similar solutions with finite time blow-up for $p\in(p_S(\sigma),p_L(\sigma))$, respectively $p\in(p_S(\sigma),\overline{p_L}(\sigma))$. Numerical evidence of the optimality of these exponents is also included.
\end{abstract}

\


\smallskip


\section{Introduction}

The semilinear parabolic equation
\begin{equation}\label{eq1.hom}
\partial_tu=\Delta u+u^p, \quad p>1,
\end{equation}
has become the classical model of a reaction-diffusion equation and has been strongly analyzed in the last half century. The main mathematical feature of its solutions is the finite time blow-up, which means the existence of a time $T\in(0,\infty)$ such that $u(t)\in L^{\infty}(\real^N)$ for any $t\in(0,T)$ but $u(T)$ becomes unbounded. Many important research works originated from the quest for understanding how and for which initial conditions finite time blow-up takes place, the monograph \cite{QS} being an excellent reference for the progress achieved by the work on this problem. As a cornerstone of this analysis, in the renowned paper by Fujita \cite{Fu66} the critical exponent $p_F=1+2/N$, now called \emph{the Fujita exponent}, is derived with the following property: for any $p\in(1,p_F]$, any non-trivial solution presents finite time blow-up, while for $p>p_F$, there are also global solutions (that is, solutions $u$ to \eqref{eq1.hom} such that $u(t)\in L^{\infty}(\real^N)$ for any $t>0$). From the study of finer properties related to finite time blow-up, such as the blow-up rates, a new critical exponent, known as \emph{the Sobolev critical exponent}, has been established:
\begin{equation}\label{pS.hom}
p_S=\left\{\begin{array}{ll}\frac{N+2}{N-2}, & {\rm if} \ N\geq3, \\ +\infty, & {\rm if} \ N\in\{1,2\}.\end{array}\right.
\end{equation}
In the well-known works by Giga and Kohn \cite{GK85, GK87}, by employing the technique of energy estimates in backward self-similar variables (see also \cite[Remark 23.4]{QS}), it is proved that the blow-up rate of solutions to Eq. \eqref{eq1.hom} for $p\in(1,p_S)$ is always given by
\begin{equation}\label{BUrate}
\|u(x,t)\|_{\infty}\sim K(T-t)^{-1/(p-1)}, \quad K=\left(\frac{1}{p-1}\right)^{1/(p-1)}, \quad {\rm as} \ t\to T,
\end{equation}
a rate which is now called the \emph{self-similar rate}. The supercritical case $p>p_S$ is much more complex, and in particular, solutions that blow up in a finite time but with a different rate than \eqref{BUrate} have been deduced for the first time by Herrero and Vel\'azquez in \cite{HV94}. Such a finite time blow-up is nowadays known as \emph{blow-up of type II} and has been classified in \cite{MM09, MM11}. A different phenomenon known as the \emph{continuation after blow-up} may hold true for $p>p_S$ as well; that is, a class of solutions that blow up at the origin at $t=T$ but afterwards they become again finite has been identified, such solutions being called sometimes \emph{peaking solutions} (see for example \cite{GV97}).

A particular class of solutions that is very useful in understanding the mathematical properties of general solutions to nonlinear diffusion equations is composed by the \textbf{self-similar solutions}. Since our interest is focused on finite time blow-up, the prototype for this phenomenon are the \emph{backward (or blow-up) self-similar solutions}, in the form
\begin{equation}\label{backwardSS.hom}
u(x,t)=(T-t)^{-1/(p-1)}f(|x|(T-t)^{-1/2}), \quad T\in(0,\infty),
\end{equation}
blowing up at time $T$ with a rate given by \eqref{BUrate}, where $f$ is called the profile of the self-similar solution $u$. It has been shown (see \cite[Chapter 23]{QS}) that these solutions represent the local behavior as $t\to T$ near any blow-up point of any general solution to Eq. \eqref{eq1.hom}. Moreover, a quest for establishing the existence and then classifying such solutions led to two new critical exponents. On the one hand, the \emph{Joseph-Lundgren exponent}
\begin{equation}\label{pJL.hom}
p_{JL}=\left\{\begin{array}{ll}1+\frac{4}{N-4-2\sqrt{N-1}}, & {\rm if} \ N\geq11,\\[1mm] +\infty, & {\rm if} \ N<11,\end{array}\right.
\end{equation}
appeared first in the study of \eqref{eq1.hom} in \cite{JL73} (see also \cite[Chapter 9]{QS}) and has the property (among others) that infinitely many self-similar solutions in the form \eqref{backwardSS.hom} exist for $p\in(p_S,p_{JL})$. On the other hand, the \emph{Lepin exponent}
\begin{equation}\label{Lepin.hom}
p_L=\left\{\begin{array}{ll}1+\frac{6}{N-10}, & {\rm if} \ N\geq11, \\[1mm] +\infty, & {\rm if} \ N<11,\end{array}\right.
\end{equation}
has been introduced in \cite{Le90} as an upper limit for the existence of blow-up self-similar solutions to Eq. \eqref{eq1.hom}. Remarkably, Mizoguchi established in a series of papers \cite{Mi04, Mi09, Mi10} that, indeed, there are no backward self-similar solutions to Eq. \eqref{eq1.hom} if $p\geq p_L$.

\medskip

\noindent \textbf{The Hardy-H\'enon equation.} In this paper, we are interested in the generalization of Eq. \eqref{eq1.hom},
\begin{equation}\label{eq1}
\partial_tu=\Delta u+|x|^{\sigma}u^p,
\end{equation}
known as the \emph{parabolic Hardy-H\'enon equation}, whose source term features an unbounded power weight $|x|^{\sigma}$. The name given to Eq. \eqref{eq1} stems from, on the one hand, the paper by H\'enon \cite[Eq. (A.6)]{He73} in which the elliptic counterpart of Eq. \eqref{eq1} is proposed in a model from astrophysics, and on the other hand from the paper by Baras and Goldstein \cite{BG84}, where Eq. \eqref{eq1} with $p=1$ and $\sigma=-2$ is proposed and existence and non-existence of solutions is established, depending on the optimal constant of Hardy's inequality. In this work, the dimension $N$ and exponents $(p,\sigma)$ are restricted to
\begin{equation}\label{range.exp}
N\geq3, \quad \sigma\in(-2,\infty), \quad p>1.
\end{equation}
Dimensions $N=1$ and $N=2$ are excluded from the analysis since, as we shall see below, the critical exponents limiting our study are finite only when $N\geq3$.

The mathematical analysis of solutions to Eq. \eqref{eq1} has become fashionable in the last decade due to the strong and interesting influence of the spatially-dependent coefficient, with a number of papers studying optimal conditions for existence of solutions, their regularity, large time behavior in some ranges of exponents, an analysis of the blow-up at infinity, examples of blow-up of type II and other qualitative properties; we quote here the works \cite{BSTW17, BS19, MS21, CIT21, CIT22, CITT24, SU24} (see also references therein), as well as the generalization of the mathematical theory to the Hardy-H\'enon equation with quasilinear diffusion in \cite{IL25}.

Similarly as for Eq. \eqref{eq1.hom}, a very important role in the mathematical analysis of Eq. \eqref{eq1} is played by the self-similar solutions presenting finite time blow-up, having the form
\begin{equation}\label{backwardSS}
u(x,t)=(T-t)^{-\alpha}f(|x|(T-t)^{-1/2}), \quad \alpha=\frac{\sigma+2}{2(p-1)}>0.
\end{equation}
The Sobolev and Joseph-Lundgren critical exponents have modified forms with respect to \eqref{pS.hom} and \eqref{pJL.hom}, namely
\begin{equation}\label{pS}
p_S(\sigma)=\begin{cases}\frac{N+2\sigma+2}{N-2}, & {\rm if} \ N\geq3,\\
\infty, & {\rm if} \ N\in\{1,2\},\end{cases}
\end{equation}
and respectively
\begin{equation}\label{pJL}
p_{JL}(\sigma)=\left\{\begin{array}{ll}\frac{(N-2)^2-2(N+\sigma)(\sigma+2)+2(\sigma+2)\sqrt{(N+\sigma)^2-(N-2)^2}}{(N-2)(N-10-4\sigma)}, & N>10+4\sigma, \\ +\infty, & N\leq 10+4\sigma.\end{array}\right.
\end{equation}
A complete classification of the blow-up self-similar solutions in the form \eqref{backwardSS} to Eq. \eqref{eq1} and with $\sigma\in(-2,\infty)$ has been achieved in \cite{FT00} in the range of exponents $1<p<p_{JL}(\sigma)$. It is thus established therein that: for $\sigma\geq0$ no blow-up self-similar solution exists for $1<p<p_S(\sigma)$, while at least one self-similar solution in the form \eqref{backwardSS} with a decreasing profile exists for $\sigma\in(-2,0)$ and $1<p<p_S(\sigma)$. In the interval $p_S(\sigma)<p<p_{JL}(\sigma)$ the outcome is more striking: for any $\sigma\in(-2,\infty)$ there are infinitely many different self-similar solutions in the form \eqref{backwardSS}. 

We also mention that the authors developed a program of understanding self-similar solutions to the Hardy-H\'enon equation with quasilinear diffusion $\Delta u^m$, $m\neq1$, and we quote here, among other results, the works \cite{IMS23, ILS24, IS22, IS25a, IS25b} as the most interesting results obtained in this study.

To the best of our knowledge, there is no analysis of self-similar solutions to Eq. \eqref{eq1} for any $\sigma\neq0$ and $p>p_{JL}(\sigma)$, and the goal of the present paper is to complete this analysis. Moreover, it was a rather big surprise to us when we realized that the outcome of our analysis, as it will be described in the paragraph devoted to the main results, \emph{strongly departs from the homogeneous equation \eqref{eq1.hom}} and that the exponents $\sigma=4k-2$ with $k\geq1$ a natural number are critical with respect to the number of different self-similar solutions. We are thus in a position to state our main results.

\medskip

\noindent \textbf{Main results.} We assume throughout this work that the conditions in \eqref{range.exp} are in force and, furthermore,
\begin{equation}\label{range.exp2}
N>10+4\sigma, \quad p>p_{JL}(\sigma).
\end{equation}
As commented in the previous paragraphs, our aim is to classify the ranges of existence, non-existence, and multiplicity of self-similar solutions in backward form \eqref{backwardSS} to Eq. \eqref{eq1}. Plugging the ansatz \eqref{backwardSS} into Eq. \eqref{eq1}, we deduce by direct calculation that the profile $f$ of a self-similar solution solves the differential equation
\begin{equation}\label{SSODE}
f''(\xi)+\frac{N-1}{\xi}f'(\xi)-\frac{1}{2}\xi f'(\xi)-\frac{\sigma+2}{2(p-1)}f(\xi)+\xi^{\sigma}f(\xi)^p=0, \quad \xi>0.
\end{equation}
Before stating our main results, let us recall that Eq. \eqref{eq1} admits in the range of exponents \eqref{range.exp}-\eqref{range.exp2} the stationary solution
\begin{equation}\label{stat.sol}
U(x)=C(\sigma)|x|^{-(\sigma+2)/(p-1)}, \quad C(\sigma)=\left[\frac{(\sigma+2)[(N-2)p-(N+\sigma)]}{(p-1)^2}\right]^{1/(p-1)}.
\end{equation}
We are now in a position to state the main theorems of this work and we start with, in our opinion, the most unexpected result.
\begin{theorem}\label{th.large}
Let $N$, $p$ satisfying \eqref{range.exp2} and $\sigma\geq2$. Then there exists at least one radially symmetric self-similar solution to Eq. \eqref{eq1} in the form \eqref{backwardSS} such that its profile satisfies
\begin{equation}\label{decay.peak}
f(0)>0, \quad f'(0)=0, \quad \lim\limits_{\xi\to\infty}\xi^{(\sigma+2)/(p-1)}f(\xi)=K\in(0,C(\sigma)).
\end{equation}
Let $k\geq1$ be a natural number. In the same conditions, if $\sigma\geq4k-2$, there are at least $k$ different self-similar solutions to Eq. \eqref{eq1} whose profiles solve \eqref{SSODE} and satisfy the conditions \eqref{decay.peak}.
\end{theorem}
This result is in striking contrast to the homogeneous case $\sigma=0$, corresponding to the usual reaction-diffusion equation Eq. \eqref{eq1.hom}. Indeed, as previously commented, Mizoguchi proved in \cite{Mi04, Mi09, Mi10} the non-existence of any blow-up self-similar solution to Eq. \eqref{eq1.hom} if $p\geq p_L$, where $p_L$ is defined in \eqref{Lepin.hom}. Another remarkable fact is that the self-similar solutions obtained in Theorem \ref{th.large} are expected to be \emph{peaking solutions}, according to the terminology of \cite{GV97}; that is, solutions that can be continued after the blow-up time $T$ by a forward, global in time, self-similar solution, as explained in \cite[Theorem 11.1]{GV97}. The difference between peaking solutions and solutions presenting complete blow-up, at least in the case of quasilinear (degenerate) diffusion, is noticed at the level of the limit equal to $K$ in \eqref{decay.peak} and depends on whether $K<C(\sigma)$ (peaking) or $K>C(\sigma)$ (complete blow-up), and it is highly expected that the same phenomenon holds true in our semilinear case. However, we refrain here from entering more deeply into this discussion. To end these comments, another remarkable fact is the multiplicity of such solutions, and that their number increases with $\sigma$, showing one more interesting effect of the presence of the spatially-dependent coefficient.

For $\sigma\in(0,2)$, a Lepin-type exponent still appears, similar to the non-weighted case $\sigma=0$. Let us thus introduce the exponent
\begin{equation}\label{Lepin}
p_L(\sigma):=\left\{\begin{array}{ll}\frac{(2-\sigma)(N+\sigma-4)}{N(2-\sigma)-2(10-\sigma)}, & {\rm if} \ N>\frac{2(10-\sigma)}{2-\sigma}, \\[1mm] +\infty, & {\rm if} \ N\leq\frac{2(10-\sigma)}{2-\sigma},\end{array}\right.
\end{equation}
noticing that $p_L(0)=p_L$ according to \eqref{Lepin.hom} and that
$$
\frac{2(10-\sigma)}{2-\sigma}-10-4\sigma=\frac{4\sigma^2}{2-\sigma}>0,
$$
hence the lower bound for the dimension in \eqref{Lepin} is compatible with the one for the lower exponent $p_{JL}(\sigma)$ defined in \eqref{pJL}. Moreover, we shall see that $p_{L}(\sigma)>p_{JL}(\sigma)$ whenever $p_L(\sigma)<\infty$.
\begin{theorem}\label{th.small}
Let $N$, $p$ satisfying \eqref{range.exp2} and $\sigma\in(0,2)$. Then, if $p_{JL}(\sigma)<p<p_L(\sigma)$, there exists at least one self-similar solution to Eq. \eqref{eq1} in backward form \eqref{backwardSS} such that its profile satisfies \eqref{decay.peak}.
\end{theorem}
We can observe that Theorem \ref{th.small} is a direct generalization to $\sigma\in(0,2)$ of the existence results valid for $\sigma=0$ established in \cite{Le90} if $p_{JL}<p<p_L$. We stress here once more that the self-similar solutions obtained in Theorem \ref{th.small} are expected to be peaking solutions, in the sense of continuation after the blow-up time explained in \cite{GV97}.

\medskip 

\noindent \textbf{Remark.} We can give an alternative form of the outcome of Theorems \ref{th.large} and \ref{th.small} by expressing $\sigma$ in terms of $N$. Indeed, if $N>10$ and
\begin{equation}\label{sigmaN}
\sigma\geq\frac{2(N-10)}{N-2},
\end{equation}
then there exists at least one self-similar solution in the backward form \eqref{backwardSS}. On the contrary, if the inequality sign is reversed in \eqref{sigmaN}, then the existence is limited to the interval $p\in(p_{JL}(\sigma),p_{L}(\sigma))$. Note that \eqref{sigmaN} is fulfilled for any $\sigma\geq2$. We observe that the previous approach produces existence of solutions also for negative values of $\sigma$, more precisely, when
$$
-\frac{2(N-10)}{N+6}<\sigma<0, \quad p_{JL}(\sigma)<p<p_L(\sigma),
$$
but the above condition is far from optimal for $\sigma<0$ and it will be improved in the forthcoming Theorem \ref{th.neg}.

\medskip

We are left with the negative range $\sigma\in(-2,0)$, featuring different properties from the previous ones, since in this case the weight $|x|^{\sigma}$ becomes singular at the origin instead of unbounded as $|x|\to\infty$. Similarly as in the previous case, let us introduce the following Lepin-type exponent
\begin{equation}\label{Lepin.neg}
\overline{p_L}(\sigma):=\left\{\begin{array}{ll}\frac{\sigma(N-2+\sigma)}{\sigma(N-2)+4}, & {\rm if} \ N>\frac{2\sigma-4}{\sigma},\\
\\[1mm] +\infty, & {\rm if} \ N\leq\frac{2\sigma-4}{\sigma},\end{array}\right.
\end{equation}
observing that
$$
\frac{2\sigma-4}{\sigma}-10-4\sigma=-\frac{4(\sigma+1)^2}{\sigma}>0.
$$
Moreover, we shall see that $\overline{p_L}(\sigma)>p_{JL}(\sigma)$ whenever $\overline{p_L}(\sigma)<\infty$. For $\sigma\in(-2,0)$ we prove the existence of a non-peaking self-similar solution with a decreasing profile for any exponent $p$ satisfying $p_{JL}(\sigma)<p<\overline{p_L}(\sigma)$, extending thus the result established for $p<p_{JL}(\sigma)$ in \cite{FT00}.
\begin{theorem}\label{th.neg}
Let $N$, $p$ satisfying \eqref{range.exp2} and $\sigma\in(-2,0)$ be such that
\begin{equation}\label{interm13}
N\geq\frac{2(2-\sigma)}{\sigma+2}.
\end{equation}
Then, if $p_{JL}(\sigma)<p<\overline{p_L}(\sigma)$, there exists at least one self-similar solution to Eq. \eqref{eq1} in backward form \eqref{backwardSS} having a decreasing profile $f$ and such that
\begin{equation}\label{decay.neg}
f(0)>0, \quad \lim\limits_{\xi\to\infty}\xi^{(\sigma+2)/(p-1)}f(\xi)=K\in(C(\sigma),\infty).
\end{equation}
\end{theorem}
A simple calculation gives that, for $\sigma\geq-1$, the condition \eqref{interm13} is automatically implied by the condition $N>10+4\sigma$ contained in the definition of $p_{JL}(\sigma)$, since
$$
10+4\sigma-\frac{2(2-\sigma)}{\sigma+2}=\frac{4(\sigma+1)(\sigma+4)}{\sigma+2}\geq0.
$$
For $\sigma\in(-2,-1)$, we have
$$
\frac{2(2-\sigma)}{\sigma+2}-\frac{2\sigma-4}{\sigma}=\frac{4(2-\sigma)(\sigma+1)}{\sigma(\sigma+2)}>0,
$$
hence the condition \eqref{interm13} directly implies that the exponent $\overline{p_L}(\sigma)$ is finite. Observe also that the limit as $\xi\to\infty$ in \eqref{decay.neg} is now bigger than the constant $C(\sigma)$ of the stationary solution \eqref{stat.sol}, a fact that should induce complete blow-up as $t\to T$. Another difference with respect to the range $\sigma>0$ is that the profiles of the solutions existing for $\sigma<0$ are decreasing, while the profiles of the solutions existing for $\sigma>0$ have at least a strict maximum at a positive value of $\xi$, as one can deduce from the fact that $f(0)>0$, $f'(0)=0$ and \eqref{SSODE} imply that
$$
f''(0)=\frac{\sigma+2}{2N(p-1)}f(0)>0.
$$
Finally, if $\sigma\in(-2,0)$, the $C^2$-regularity at $x=0$ of the self-similar solutions, obviously valid when $\sigma>0$, might be lost and the optimal regularity of the self-similar solutions considered in Theorem \ref{th.neg} relies on suitable Sobolev or H\"older spaces, still sufficient for a weak solution, as explained in \cite[Theorem 1.2 and Section 3.3]{IL25b}.

\medskip

\noindent \textbf{Discussion about non-existence.} Theorems \ref{th.small} and \ref{th.neg} left out the range of exponents $p\geq p_L(\sigma)$ when $\sigma\in(0,2)$ and $p_L(\sigma)$ defined in \eqref{Lepin}, respectively $p\geq \overline{p_L}(\sigma)$ when $\sigma\in(-2,0)$ and $\overline{p_L}(\sigma)$ defined in \eqref{Lepin.neg}. Our strong expectation is that non-existence of self-similar solutions in the form \eqref{backwardSS} holds true in these ranges, exactly as it occurs for $\sigma=0$ and $p\geq p_L$, as proved in \cite{Mi04, Mi09, Mi10}. We show some numerical experiments (together with their interpretation) enforcing this idea that our range of existence is sharp for both $\sigma\in(0,2)$ and $\sigma\in(-2,0)$. A possible approach to proving such non-existence results is to try to adapt to Eq. \eqref{eq1} the technique employed in \cite{Mi09} for the simpler equation when $\sigma=0$, noticing that a part of the (numerous) preparatory lemmas therein can be proved in an alternative way by employing our phase-space technique. However, the variable coefficient $|x|^{\sigma}$ might introduce a lot of technical difficulties in establishing analogous estimates to the ones in \cite{Mi09}. In order to keep this paper reasonably short, we refrain from addressing the question of proving non-existence results in this work.

\medskip

\noindent \textbf{Structure of the paper.} We start with a rather technical Section \ref{sec.syst}, where an alternative formulation of the differential equation \eqref{SSODE} by a transformation to an autonomous dynamical system is introduced, and its equilibrium points (both finite and at infinity) are analyzed. Some preparatory results for the global analysis of the dynamical system are also given at the end of Section \ref{sec.syst}. The next Section \ref{sec.large} is devoted to the proof of Theorems \ref{th.large} and \ref{th.small}, both based on a backward shooting technique (that is, a shooting in the opposite sense of the flow, starting from the tail behavior) performed in the alternative formulation introduced in Section \ref{sec.syst}. We mention here that the main calculation allowing the reader to foresee why existence holds true always if $\sigma\geq2$ and how the exponent $p_L(\sigma)$ arises, is given in the first part of Section \ref{sec.large} and can be followed directly, even if one skips the dynamical system in Section \ref{sec.syst}. The last section of the paper, Section \ref{sec.neg}, deals with the interval $\sigma\in(-2,0)$ and proves Theorem \ref{th.neg}, employing this time a direct shooting technique (different from the proofs in Section \ref{sec.large}). Some numerical experiments confirming and visually supporting the proofs of the main results and at the same time suggesting the sharpness of our existence results in Theorems \ref{th.small} and \ref{th.neg} are given at the end of the corresponding proofs.

\section{An autonomous dynamical system}\label{sec.syst}

The proofs of the main results follow from a combination of techniques working directly on the differential equation \eqref{SSODE} and on an alternative, transformed formulation of it in the form of an autonomous dynamical system. Thus, before going to the proofs, in this preparatory section we introduce and analyze the above mentioned dynamical system. We recall that we assume, throughout the paper, that $N\geq3$ and $p>p_{JL}(\sigma)$.

Starting from \eqref{SSODE}, we thus perform the following change of variable
\begin{equation}\label{PSchange}
X(\eta):=\frac{\sigma+2}{2(p-1)}\xi^2, \quad Y(\eta):=\frac{\xi f'(\xi)}{f(\xi)}, \quad Z(\eta):=\xi^{\sigma+2}f(\xi)^{p-1},
\end{equation}
where the new independent variable $\eta$ is given by $\eta:=\ln\,\xi$. Thus, the interval $\xi\in(0,\infty)$ is mapped into $\eta\in\real$. We obtain by direct algebraic manipulations that the new variables $(X,Y,Z)(\eta)$ solve the following autonomous differential system
\begin{equation}\label{PSsyst}
\left\{\begin{array}{ll}X'=2X, \\ Y'=X-(N-2)Y-Z-Y^2+\frac{p-1}{\sigma+2}XY, \\ Z'=Z(\sigma+2+(p-1)Y),\end{array}\right.
\end{equation}
where the primes indicate derivatives with respect to the variable $\eta$. Let us remark that the system \eqref{PSsyst} is also derived by letting $m=1$ in the dynamical system employed with success in \cite[Section 2]{IS25a}, and thus some technical parts in the forthcoming analysis that are identical as in the quoted reference will be only sketched here. We also observe that we are only interested in the open region of the phase space
$$
(X,Y,Z)\in\mathcal{R}_+:=(0,\infty)\times\real\times(0,\infty),
$$
and that the planes $\{X=0\}$ and $\{Z=0\}$ limiting $\mathcal{R}_+$ are invariant for the system \eqref{PSsyst}.

\subsection{Critical points}\label{subsec.crit}

The system \eqref{PSsyst} has three critical points:
$$
P_0=(0,0,0), \quad P_1=(0,-(N-2),0), \quad P_2=\left(0,-\frac{\sigma+2}{p-1},Z_0\right),
$$
where
\begin{equation}\label{Z0}
Z_0:=\frac{(\sigma+2)[p(N-2)-(N+\sigma)]}{(p-1)^2},
\end{equation}
and the local analysis of the flow of the system in a neighborhood of them is given below.
\begin{lemma}\label{lem.P0}
The critical point $P_0$ is a saddle point with a one-dimensional stable manifold included in the $Y$-axis and a two-dimensional unstable manifold. The trajectories contained in the unstable manifold enter the region $\mathcal{R}_+$ and represent, after undoing the change of variable \eqref{PSchange}, profiles $f$ such that $f(0)>0$ and, if $\sigma>0$, then also $f'(0)=0$.
\end{lemma}
\begin{proof}
The linearization of the system \eqref{PSsyst} in a neighborhood of the origin has the matrix
$$
M(P_0)=\left(
         \begin{array}{ccc}
           2 & 0 & 0 \\
           1 & -(N-2) & -1 \\
           0 & 0 & \sigma+2 \\
         \end{array}
       \right),
$$
with two positive eigenvalues $\lambda_1=2$, $\lambda_3=\sigma+2$ and one negative eigenvalue $\lambda_2=-(N-2)$. Since the eigenvector $e_2=(0,1,0)$ corresponding to the eigenvalue $\lambda_2$ is contained in the $Y$-axis, which is invariant for the system \eqref{PSsyst}, the uniqueness of the stable manifold \cite[Theorem 3.2.1]{GH} gives that it is contained completely in this $Y$-axis. On the contrary, the unstable manifold is tangent to the plane spanned by the two eigenvectors
\begin{equation}\label{eigen.P0}
e_1=(N,1,0), \qquad e_3=(0,1,-(N+\sigma)).
\end{equation}
It readily follows from the first and third equations of the system \eqref{PSsyst} that
\begin{equation}\label{interm1}
\lim\limits_{\eta\to-\infty}\frac{Z(\eta)}{X(\eta)^{(\sigma+2)/2}}=C\in(0,\infty),
\end{equation}
and thus $f(0)>0$ by undoing the change of variable \eqref{PSchange} on \eqref{interm1}. Moreover, \eqref{interm1} implies in particular that either $Z(\eta)$ or $X(\eta)$ dominate as $\eta\to-\infty$, depending on $\sigma$, as follows:
\begin{equation}\label{interm3}
\lim\limits_{\eta\to-\infty}\frac{Z(\eta)}{X(\eta)}=\left\{\begin{array}{ll}0, & {\rm if} \ \sigma>0,\\ +\infty, & {\rm if} \ \sigma<0.\end{array}\right.
\end{equation}
We then deduce from \eqref{interm3} that, for $\sigma>0$, the second equation of the system \eqref{PSsyst} reads
$$
Y'(\eta)=X(\eta)-(N-2)Y(\eta)+o(|(X,Y)(\eta)|), \quad {\rm as} \ \eta\to-\infty,
$$
while for $\sigma<0$ it reads
$$
Y'(\eta)=-(N-2)Y(\eta)-Z(\eta)+o(|(X,Y)(\eta)|), \quad {\rm as} \ \eta\to-\infty.
$$
The previous approximations lead to the following local expansions as $\eta\to-\infty$:
\begin{equation}\label{interm2}
Y(\eta)\sim\left\{\begin{array}{ll}\frac{X(\eta)}{N}, & {\rm if} \ \sigma>0, \\[1mm] \frac{-Z(\eta)}{N+\sigma}, & {\rm if} \ \sigma\in(-2,0),\end{array}\right.
\end{equation}
and the local behavior up to a second order is readily deduced from \eqref{interm2} by undoing the change of variable \eqref{PSchange}. In particular, for $\sigma>0$, we get from \eqref{interm2} that
\begin{equation}\label{interm4}
\lim\limits_{\xi\to0}\frac{f'(\xi)}{\xi f(\xi)}=\frac{N(\sigma+2)}{2(p-1)},
\end{equation}
which gives $f'(0)=0$, completing the proof.
\end{proof}

\noindent \textbf{Remark.} Let us notice that we can obtain a more precise representation of the unstable manifold of $P_0$. Fixing $C>0$ from \eqref{interm1}, keeping only the dominating terms in a first approximation in the second equation of \eqref{PSsyst}, and replacing $Z(\eta)$ from \eqref{interm1}, we approximate the system by
$$
X'=2X, \quad Y'=X-(N-2)Y-CX^{(\sigma+2)/2}.
$$
By integrating this system, we deduce that the trajectories contained in the unstable manifold of $P_0$ have the local expansion
\begin{equation}\label{lC}
Y(\eta)\sim\frac{1}{N}X(\eta)-\frac{C}{N+\sigma}X^{(\sigma+2)/2}(\eta)=\frac{X(\eta)}{N}-\frac{Z(\eta)}{N+\sigma}, \quad {\rm as} \ \eta\to-\infty.
\end{equation}
The Hartman-Grobman Theorem implies readily that, for any given $C\in(0,\infty)$, there is a unique trajectory satisfying \eqref{interm1} and \eqref{lC} with the given constant $C$, and this trajectory will be denoted by $l_C$ in the sequel. By convention, we denote by $l_0$ the trajectory contained in the unstable manifold of $P_0$ and in the invariant plane $\{Z=0\}$ and by $l_{\infty}$ the trajectory contained in the unstable manifold of $P_0$ and in the invariant plane $\{X=0\}$.

An integration over $(0,\xi)$ in \eqref{interm4} (with $\xi>0$ small) leads to a more precise local expansion as $\xi\to0$, that is
\begin{equation}\label{exp.pos}
f(\xi)\sim f(0)\exp\left[\frac{N(\sigma+2)}{4(p-1)}\xi^2\right].
\end{equation}
In a similar way (see for example \cite[Lemma 2.1]{IS25b} or \cite[Section 3.3]{IL25b}), one can derive the local expansion as $\xi\to0$ of the profiles contained in the unstable manifold of $P_0$ for $\sigma<0$; that is,
\begin{equation}\label{exp.neg}
f(\xi)\sim\left[K+\frac{p-1}{(N+\sigma)(\sigma+2)}\xi^{\sigma+2}\right]^{-1/(p-1)}, \quad K=f(0)^{1-p}\in(0,\infty).
\end{equation}

The next result gathers the local analysis near the critical points $P_1$ and $P_2$.
\begin{lemma}\label{lem.P1P2}
(a) The critical point $P_1$ is a saddle point with a two-dimensional unstable manifold contained in the invariant plane $\{Z=0\}$ and a one-dimensional stable manifold contained in the invariant plane $\{X=0\}$. No trajectories connect $P_1$ to points in the region $\mathcal{R}_+$.

(b) The critical point $P_2$ is a saddle point with a two-dimensional stable manifold contained in the invariant plane $\{X=0\}$ and a one-dimensional unstable manifold containing the stationary solution \eqref{stat.sol}.
\end{lemma}
\begin{proof}
(a) The linearization of the system \eqref{PSsyst} in a neighborhood of the critical point $P_1$ has the matrix
$$
M(P_1)=\left(
         \begin{array}{ccc}
           2 & 0 & 0 \\[1mm]
           \frac{N+\sigma-p(N-2)}{\sigma+2} & N-2 & -1 \\[1mm]
           0 & 0 &  N+\sigma-p(N-2)\\
         \end{array}
       \right),
$$
with eigenvalues $\lambda_1=2$, $\lambda_2=N-2$ (both positive) and $\lambda_3=N+\sigma-p(N-2)$ (negative, since $p>p_{JL}(\sigma)$) and corresponding eigenvectors $e_1$, $e_2$ contained in the plane $\{Z=0\}$ and $e_3$ contained in the plane $\{X=0\}$ (as one can check by direct calculation). The stable manifold theorem \cite[Section 2.7]{Pe}, the uniqueness of the stable and unstable manifolds \cite[Theorem 3.2.1]{GH} and the invariance of the planes $\{X=0\}$ and $\{Z=0\}$ imply that the whole unstable manifold of $P_1$ is contained in the plane $\{Z=0\}$ and the whole stable manifold of $P_1$ is contained in the plane $\{X=0\}$, completing the proof.

\medskip

(b) The linearization of the system \eqref{PSsyst} in a neighborhood of the critical point $P_2$ has the matrix
$$
M(P_2)=\left(
         \begin{array}{ccc}
           2 & 0 & 0 \\[1mm]
           0 & \frac{(N-2)(p_S(\sigma)-p)}{p-1} & -1 \\[1mm]
           0 & \frac{(\sigma+2)[p(N-2)-(N+\sigma)]}{p-1} & 0 \\
         \end{array}
       \right),
$$
where $p_S(\sigma)$ is defined in \eqref{pS}. Since $p>p_{JL}(\sigma)>p_S(\sigma)$, we infer that the eigenvalues of $M(P_2)$ satisfy the conditions $\lambda_1=2$ and
$$
\lambda_2+\lambda_3=\frac{(N-2)(p_S(\sigma)-p)}{p-1}<0, \quad \lambda_2\lambda_3=\frac{(\sigma+2)[p(N-2)-(N+\sigma)]}{p-1}>0,
$$
thus both $\lambda_2$ and $\lambda_3$ are either negative or complex with negative real parts. Moreover, a rather tedious but direct calculation gives that $p_{JL}(\sigma)$ is obtained exactly by equating
$$
(\lambda_2+\lambda_3)^2-4\lambda_2\lambda_3=0
$$
(see \cite{Le90} for the similar calculation with $\sigma=0$), and thus for $p>p_{JL}(\sigma)$ both $\lambda_2$ and $\lambda_3$ are real negative numbers. An inspection of the eigenvectors proves that the stable manifold spanned by the eigenvectors $e_2$ and $e_3$ corresponding to the eigenvalues $\lambda_2$ and $\lambda_3$ is contained in the invariant plane $\{X=0\}$. The uniqueness of the one-dimensional unstable manifold and the fact that, by applying the change of variable \eqref{PSchange} to the stationary solution \eqref{stat.sol} we obtain the line of equation \begin{equation}\label{stat.sol.plane}
Y=-\frac{\sigma+2}{p-1}, \quad Z=Z(P_2):=\frac{(\sigma+2)[p(N-2)-(N+\sigma)]}{(p-1)^2},
\end{equation}
imply that the unique trajectory contained in the unstable manifold of $P_2$ is the line \eqref{stat.sol.plane}, as claimed.
\end{proof}

\subsection{Analysis at infinity: the Poincar\'e hypersphere}\label{subsec.inf}

The previous local analysis is completed by the analysis of the critical points at infinity of the phase space associated to the system \eqref{PSsyst}. To this end, we perform a compactification of the space by setting
$$
X=\frac{\overline{X}}{W}, \qquad Y=\frac{\overline{Y}}{W}, \qquad Z=\frac{\overline{Z}}{W}
$$
and thus constructing the Poincar\'e hypersphere in variables $(\overline{X},\overline{Y},\overline{Z},W)$. Following the theory in \cite[Theorem 4, Section 3.10]{Pe}, the critical points at infinity of the system \eqref{PSsyst} are thus given by the solutions of the system
\begin{equation*}
\left\{\begin{array}{ll}\frac{1}{\sigma+2}\overline{X}\overline{Y}[(p-1)\overline{X}-(\sigma+2)\overline{Y}]=0,\\
(p-1)\overline{X}\overline{Z}\overline{Y}=0,\\
\frac{1}{\sigma+2}\overline{Y}\overline{Z}[p(\sigma+2)\overline{Y}-(p-1)\overline{X}]=0,\end{array}\right.
\end{equation*}
together with the condition of belonging to the equator of the hypersphere, that is, $W=0$ and the additional equation $\overline{X}^2+\overline{Y}^2+\overline{Z}^2=1$. Recalling that $\overline{X}\geq0$ and $\overline{Z}\geq0$, we obtain the following critical points:
\begin{equation*}
\begin{split}
&Q_1=(1,0,0,0), \ \ Q_{2,3}=(0,\pm1,0,0), \ \ Q_4=(0,0,1,0), \ \ Q_{\gamma}=\left(\gamma,0,\sqrt{1-\gamma^2},0\right),\\
&Q_5=\left(\frac{\sigma+2}{\sqrt{(\sigma+2)^2+(p-1)^2}},\frac{p-1}{\sqrt{(\sigma+2)^2+(p-1)^2}},0,0\right),
\end{split}
\end{equation*}
with $\gamma\in(0,1)$. In order to analyze the flow of the system in a neighborhood of the critical points $Q_1$, $Q_{\gamma}$ and $Q_5$ (all them characterized by a nonzero first coordinate), we employ the system based on the projection on the $X$ variable according to \cite[Theorem 5(a), Section 3.10]{Pe}, which in our case becomes,
\begin{equation}\label{PSinf1}
\left\{\begin{array}{ll}\dot{x}=-2x^2,\\
\dot{y}=-y^2+\frac{p-1}{\sigma+2}y+x-Nxy-xz,\\
\dot{z}=z[(p-1)y+\sigma x],\end{array}\right.
\end{equation}
where the variables expressed in lowercase letters are obtained from the original ones by the following change of variable
\begin{equation}\label{change2}
x=\frac{1}{X}, \qquad y=\frac{Y}{X}, \qquad z=\frac{Z}{X},
\end{equation}
and the independent variable with respect to which derivatives are taken in \eqref{PSinf1} is defined by
\begin{equation}\label{indep2}
\frac{d\eta_1}{d\xi}=\frac{X(\xi)}{\xi}=\frac{\alpha}{m}\xi, \quad {\rm that \ is} \quad \eta_1=\frac{\alpha}{2m}\xi^2.
\end{equation}
The critical points $Q_1$, $Q_{\gamma}$ and $Q_5$ are identified in the system \eqref{PSinf1}, respectively, by the critical points $(0,0,0)$, $(0,0,\kappa)$ with $\kappa:=\sqrt{1-\gamma^2}/\gamma\in(0,\infty)$, and $(0,(p-1)/(\sigma+2),0)$. We analyze below the flow of the system in a neighborhood of these points, starting from $Q_1$, noticing that this analysis is, in many steps of it, a particular case of the one performed for the Hardy-H\'enon equation with quasilinear diffusion in the recent work \cite{IS25a}.
\begin{lemma}\label{lem.Q1}
The critical point $Q_1$ of the system \eqref{PSinf1} is non-hyperbolic, with a unique two-dimensional center manifold and a one-dimensional unstable manifold contained in the $y$-axis. The flow on the center manifold has a stable direction, and the trajectories contained in the center manifold correspond in variables $(f,\xi)$ to profiles such that
\begin{equation}\label{decay.Q1}
\lim\limits_{\xi\to\infty}\xi^{(\sigma+2)/(p-1)}f(\xi)=K\in(0,\infty).
\end{equation}
\end{lemma}
\begin{proof}
The linearization of the system \eqref{PSinf1} in a neighborhood of the origin has the matrix
$$M(Q_1)=\left(
         \begin{array}{ccc}
           0 & 0 & 0 \\[1mm]
           1 & \frac{p-1}{\sigma+2} & 0 \\[1mm]
           0 & 0 & 0 \\
         \end{array}
       \right).
$$
It is then obvious that the one-dimensional unstable manifold is contained in the $y$-axis, owing to its invariance and the eigenvector $e_2=(0,1,0)$ corresponding to the only positive eigenvalue $\lambda_2=(p-1)/(\sigma+2)$. We analyze next the center manifolds of $Q_1$. To this end, we further introduce the change of variable
\begin{equation}\label{cmchange}
w=\frac{p-1}{\sigma+2}y+x,
\end{equation}
which allows us to transform the system \eqref{PSinf1} into a system written in the canonical form for the center manifold theory according to \cite{Carr} (see also \cite[Section 2.12]{Pe}), that is,
\begin{equation}\label{interm9}
\left\{\begin{array}{ll}\dot{x}=-2x^2,\\[1mm]
\dot{w}=\frac{p-1}{\sigma+2}w-\frac{\sigma+2}{p-1}w^2-\frac{p-1}{\sigma+2}xz+\frac{2(\sigma+2)-N(p-1)}{p-1}xw-\frac{\sigma+2p-N(p-1)}{p-1}x^2,\\[1mm]
\dot{z}=-2xz+(\sigma+2)zw.\end{array}\right.
\end{equation}
Following \cite[Section 2.5]{Carr}, we then look for an approximation of the center manifolds
$$
w(x,z)=ax^2+bxz+cz^2+O(|(x,z)|^3).
$$
Replacing the above ansatz into the general equation of the center manifold and employing the vector field of the system \eqref{interm9} and \cite[Theorem 3, Section 2.5]{Carr} (see also \cite[Lemma 2.1]{IMS23} for detailed calculations), we obtain the following approximation for the center manifold of $Q_1$ after undoing the change of variable \eqref{cmchange}:
\begin{equation}\label{cmf}
y=-\frac{\sigma+2}{p-1}x+\frac{(\sigma+2)^2[N+\sigma-p(N-2)]}{(p-1)^3}x^2+\frac{\sigma+2}{p-1}xz+xO(|(x,z)|^2),
\end{equation}
observing by a simple induction that the remainder is a multiple of $x$, due to the lack of pure powers of $z$ in the system \eqref{interm9}. We then infer from \eqref{cmf} and \cite[Theorem 2, Section 2.4]{Carr} that the flow on the center manifolds is given by the reduced system
\begin{equation}\label{red.syst}
\left\{\begin{array}{ll}\dot{x}&=-2x^2+x^2O(|(x,z)|),\\[1mm]
\dot{z}&=-2xz+xO(|(x,z)|^2),\end{array}\right.
\end{equation}
hence the direction of the flow on any center manifold is stable (towards the critical point $Q_1$). The latter fact, together with \cite[Theorem 3.2']{Sij}, gives the uniqueness of the center manifold. We further get by integration in \eqref{red.syst} that
\begin{equation*}
\lim\limits_{\eta_1\to\infty}\frac{z(\eta_1)}{x(\eta_1)}=k>0
\end{equation*}
and then, by \eqref{change2} and \eqref{indep2},
\begin{equation}\label{interm5}
\lim\limits_{\xi\to\infty}Z(\xi)=k\in(0,\infty).
\end{equation}
We notice that \eqref{interm5} implies the decay \eqref{decay.Q1} after undoing the change of variable \eqref{PSchange}, as claimed.
\end{proof}

\noindent \textbf{Remark.} The previous proof actually shows that, for any $k\in(0,\infty)$, there is a unique trajectory entering $Q_1$ on its center manifold such that \eqref{interm5} holds true. Let us denote by $\mathcal{C}_k$ this trajectory, corresponding to a limit $k\in(0,\infty)$ in \eqref{interm5}. We observe that the stationary solution \eqref{stat.sol}, represented in the phase space as the line \eqref{stat.sol.plane}, is in fact the trajectory $\mathcal{C}_{Z_0}$. Thus, the local behaviors \eqref{decay.peak} and \eqref{decay.neg} correspond to trajectories $\mathcal{C}_{k}$ with $k<Z_0$, respectively $k>Z_0$. Let us finally note that \eqref{cmf} and the already established decay \eqref{decay.Q1} imply (after undoing the changes of variable \eqref{change2} and \eqref{PSchange} for $y$, respectively $Y$) that the profiles contained in the trajectories entering $Q_1$ also satisfy
\begin{equation}\label{beh.Q1.deriv}
\lim\limits_{\xi\to\infty}\xi^{(\sigma+2)/(p-1)+1}f'(\xi)\in(0,\infty).
\end{equation}

We continue our analysis with the critical point $Q_5$.
\begin{lemma}\label{lem.Q5}
The critical point $Q_5$ is a non-hyperbolic critical point for the system \eqref{PSinf1}, with a one-dimensional stable manifold, a one-dimensional unstable manifold and one-dimensional center manifolds with a stable direction of the flow. The unstable manifold is contained in the invariant plane $\{x=0\}$, while the stable manifold and all the center manifolds generate a two-dimensional center-stable manifold fully contained in the invariant plane $\{z=0\}$.
\end{lemma}
\begin{proof}
The linearization of the system \eqref{PSinf1} in a neighborhood of the critical point $Q_5=(0,(p-1)/(\sigma+2),0)$ has the matrix
$$
M(Q_5)=\left(
  \begin{array}{ccc}
    0 & 0 & 0 \\[1mm]
    1-\frac{N(p-1)}{\sigma+2} & -\frac{p-1}{\sigma+2} & 0 \\[1mm]
    0 & 0 & \frac{(p-1)^2}{\sigma+2} \\
  \end{array}
\right),
$$
with eigenvalues $\lambda_1=0$, $\lambda_2=-(p-1)/(\sigma+2)<0$, $\lambda_3=(p-1)^2/(\sigma+2)>0$, and corresponding eigenvectors
$$
e_1=\left(\frac{p-1}{\sigma+2},1-\frac{N(p-1)}{\sigma+2},0\right), \quad e_2=(0,1,0), \quad e_3=(0,0,1).
$$
The uniqueness of the stable and unstable manifolds following from \cite[Theorem 3.2.1]{GH}, together with the invariance of the planes $\{x=0\}$ and $\{z=0\}$ and the form of the eigenvectors $e_2$ and $e_3$ give that the stable and unstable manifold are unique and contained, respectively, in the $y$-axis and in the plane $\{x=0\}$. An analysis of the center manifolds follow the same ideas as in the proof of Lemma \ref{lem.Q1} by successively setting
$$
t:=y-\frac{p-1}{\sigma+2}, \quad w:=-\frac{p-1}{\sigma+2}t+[N+\sigma-p(N-2)]x,
$$
in order to translate $Q_5$ to the origin and then derive the local approximation of the equation of the center manifold. However, we omit these calculations here, since it is obvious that on any center manifold, the reduced equation giving the direction of the flow is given by $\dot{x}=-2x^2$, which implies the stable direction of the flow on the trajectories. The (infinitely many) center manifolds and the stable manifold generate a two-dimensional center-stable manifold tangent to the subspace of the plane $\{z=0\}$ spanned by the eigenvectors $e_1$ and $e_2$, and the invariance of the plane $\{z=0\}$ implies that this center-stable manifold actually belongs to the plane $\{z=0\}$, completing the proof.
\end{proof}
We continue our analysis with the critical points $Q_{\gamma}$, identified as $(0,0,\kappa)$ in the system \eqref{PSinf1}, with
\begin{equation*}
\kappa=\kappa(\gamma):=\frac{\sqrt{1-\gamma^2}}{\gamma}\in(0,\infty).
\end{equation*}
We give below the local behavior around these points for the sake of completeness, as $Q_{\gamma}$ will not play an important role in the forthcoming analysis. 
\begin{lemma}\label{lem.Qg}
For
\begin{equation}\label{eq.gamma}
\gamma=\gamma_0:=\frac{\alpha(p-1)}{\sqrt{1+\alpha^2(p-1)^2}}, \quad \alpha=\frac{\sigma+2}{2(p-1)},
\end{equation}
the critical point $Q_{\gamma_0}$ admits a unique trajectory entering it and coming from the finite part of the phase space associated to the system \eqref{PSsyst}, corresponding to profiles with the behavior
\begin{equation}\label{beh.Qg}
\lim\limits_{\xi\to\infty}\xi^{\sigma/(p-1)}f(\xi)=\left(\frac{1}{p-1}\right)^{1/(p-1)},
\end{equation}
which is either decaying to zero or unbounded depending on the sign of $\sigma\in(-2,\infty)\setminus\{0\}$, while for $\sigma=0$ the trajectory entering $Q_{\gamma_0}$ corresponds to the constant profile
\begin{equation*}
f(\xi)=\left(\frac{1}{p-1}\right)^{1/(p-1)}.
\end{equation*}
For $\gamma\in(0,1)$ with $\gamma\neq\gamma_0$, all the trajectories connecting to or from $Q_{\gamma}$ are fully contained in the plane $\{x=0\}$.
\end{lemma}
The proof is rather technical, but completely identical to the one of \cite[Lemma 3.2]{IS25a} by simply letting $m=1$ in the calculations therein, thus we omit it here and refer the reader to the mentioned reference.

We go on with the local analysis of the critical points $Q_2$ and $Q_3$, which is performed by means of a projection on the $Y$-variable according to \cite[Theorem 5(b), Section 3.10]{Pe}. We thus obtain the system
\begin{equation}\label{PSinf2}
\left\{\begin{array}{ll}\pm\dot{\widetilde{x}}=-\widetilde{x}-N\widetilde{x}\widetilde{w}+\frac{p-1}{\sigma+2}\widetilde{x}^2+\widetilde{x}^2\widetilde{w}-\widetilde{x}\widetilde{z}\widetilde{w},\\[1mm]
\pm\dot{\widetilde{z}}=-p\widetilde{z}-(N+\sigma)\widetilde{z}\widetilde{w}+\frac{p-1}{\sigma+2}\widetilde{x}\widetilde{z}+\widetilde{x}\widetilde{z}\widetilde{w}-\widetilde{z}^2\widetilde{w},\\[1mm]
\pm\dot{\widetilde{w}}=-\widetilde{w}-(N-2)\widetilde{w}^2+\frac{p-1}{\sigma+2}\widetilde{x}\widetilde{w}+\widetilde{x}\widetilde{w}^2-\widetilde{z}\widetilde{w}^2,\end{array}\right.
\end{equation}
where the new variables $\widetilde{x}$, $\widetilde{z}$, $\widetilde{w}$ are obtained from the original variables of the system \eqref{PSsyst} by the following change of variables
\begin{equation}\label{change3}
\widetilde{x}=\frac{X}{Y}, \qquad \widetilde{z}=\frac{Z}{Y}, \qquad \widetilde{w}=\frac{1}{Y},
\end{equation}
while the new independent variable $\eta_2$ of the system \eqref{PSinf2} is defined implicitly by
$$
\frac{d\eta_2}{d\xi}=\frac{Y(\xi)}{\xi}.
$$
The plus and minus signs in \eqref{PSinf2} correspond to the direction of the flow in a neighborhood of the points $Q_2$ and $Q_3$ and it is easy to see that the minus sign applies to $Q_2$ and the plus sign applies to $Q_3$, while the critical points $Q_2$, respectively $Q_3$, are mapped into the origin of \eqref{PSinf2} according to \cite[Theorem 5(b), Section 3.10]{Pe}.
\begin{lemma}\label{lem.Q23}
The critical points $Q_2$ and $Q_3$ are, respectively, an unstable node and a stable node. The trajectories going out of $Q_2$ correspond to profiles $f(\xi)$ such that there exists $\xi_0\in(0,\infty)$ and $\delta>0$ for which
\begin{equation}\label{beh.Q2}
f(\xi_0)=0, \qquad f'(\xi_0)=C>0, \qquad f>0 \ {\rm on} \ (\xi_0,\xi_0+\delta),
\end{equation}
while the trajectories entering the stable node $Q_3$ correspond to profiles $f(\xi)$ such that there exists $\xi_0\in(0,\infty)$ and $\delta\in(0,\xi_0)$ for which
\begin{equation}\label{beh.Q3}
f(\xi_0)=0, \qquad f'(\xi_0)=-C<0, \qquad f>0 \ {\rm on} \ (\xi_0-\delta,\xi_0).
\end{equation}
\end{lemma}
\begin{proof}
The proof of the stability of the critical points is completely identical to the one of \cite[Lemma 3.3]{IS25a} (by letting $m=1$ in the calculations therein), except for the local behavior of the profiles, which we give below for $Q_3$ (the analysis being analogous for $Q_2$). We thus fix plus signs in \eqref{PSinf2} and readily deduce from the equations of the system \eqref{PSinf2} that, in a neighborhood of the origin, we have
\begin{equation}\label{interm6}
\widetilde{w}(\eta_2)\sim C_1\widetilde{x}(\eta_2), \quad |\widetilde{z}(\eta_2)|\sim C_2|\widetilde{w}(\eta_2)|^p, \quad {\rm as} \ \eta_2\to\infty,
\end{equation}
with $C_1$ and $C_2$ positive constants. The first equivalence in \eqref{interm6} implies $X(\xi)\to C_1$, which is equivalent to $\xi\to\xi_0\in(0,\infty)$, thus it gives an alternative proof of the fact that the limit $\eta_2\to\infty$ translates into $\xi\to\xi_0\in(0,\infty)$ (another proof of this fact is found in \cite[Lemma 3.3]{IS25a}). The second equivalence in \eqref{interm6} reads after undoing the change of variable \eqref{change3} as
$$
\lim\limits_{\xi\to\xi_0, \xi<\xi_0}Y(\xi)Z^{1/(p-1)}(\xi)=C_3<0,
$$
which is equivalent to
$$
\lim\limits_{\xi\to\xi_0}f'(\xi)=C_3\xi_0^{-(\sigma+2)/(p-1)-1}<0.
$$
Finally, the definition of $Q_3$ implies that $Y(\xi)\to-\infty$ as $\xi\to\xi_0$, and the finiteness of both $\xi$ and $f'(\xi)$ as $\xi\to\xi_0$ together with the definition of $Y$ in \eqref{PSchange} imply that necessarily $f(\xi_0)=0$, completing the proof.
\end{proof}
We are only left with the critical point $Q_4$, which is of no interest in our analysis, as it follows from the following result.
\begin{lemma}\label{lem.Q4}
There is no orbit connecting the critical point $Q_4$ to the finite part of the phase space associated to the system \eqref{PSsyst}.
\end{lemma}
\begin{proof}
Following similar ideas as in our previous works \cite{IS25a, IMS25}, we prove this result directly on the equation \eqref{SSODE}. Indeed, assume for contradiction that there is such a trajectory. The fact that it enters the critical point $Q_4$ is equivalent to the following limits along the trajectory:
$$
\lim\limits_{\eta\to\infty}Z(\eta)=\infty, \qquad \lim\limits_{\eta\to\infty}\frac{Z(\eta)}{X(\eta)}=\infty, \quad \lim\limits_{\eta\to\infty}\left|\frac{Z(\eta)}{Y(\eta)}\right|=\infty,
$$
which, in terms of profiles, become
\begin{equation}\label{interm10}
\xi^{\sigma}f(\xi)^{p-1}\to\infty, \qquad \xi^{\sigma+2}f(\xi)^{p-1}\to\infty,  \qquad \frac{f'(\xi)}{\xi^{\sigma+1}f(\xi)^{p}}\to0,
\end{equation}
when either $\xi\to\xi_0\in[0,\infty)$ or $\xi\to\infty$. The rest of the proof consists in showing that there is no solution to the differential equation \eqref{SSODE} satisfying \eqref{interm10}, and we refer the reader to \cite[Appendix]{IMS25} for a complete and detailed proof (by simply letting $m=1$ in the calculations and estimates therein).
\end{proof}
Having thus completed the local analysis of the system \eqref{PSsyst} in a neighborhood of all its critical points, we give next a number of preparatory result related to its global analysis.

\subsection{Some preparatory lemmas}\label{subsec.prep}

We gather in this section the global analysis of some distinguished trajectories lying in the invariant planes $\{X=0\}$ and $\{Z=0\}$ in the phase space associated to the system \eqref{PSsyst}. These trajectories will be later seen as ``limits" (or ``borders") of the manifolds of interest in the general global analysis.
\begin{lemma}\label{lem.X0}
The unique trajectory contained in the unstable manifold of the critical point $P_0$ and in the invariant plane $\{X=0\}$ (denoted by $l_{\infty}$) enters the critical point $P_2$. The unique trajectory contained in the stable manifold of the critical point $P_1$ and in the invariant plane $\{X=0\}$ comes from the critical point $Q_2$.
\end{lemma}
\begin{proof}
This proof is very similar to the one of \cite[Lemmas 5.1-5.2]{IMS25}, but we give the details here for the sake of completeness, since this result will be very important in the sequel. Let us observe that the system \eqref{PSsyst} reduces in the invariant plane $\{X=0\}$ to
\begin{equation}\label{PSred.X0}
\left\{\begin{array}{ll}Y'=-(N-2)Y-Z-Y^2, \\ Z'=Z(\sigma+2+(p-1)Y),\end{array}\right.
\end{equation}
and both points $P_0=(0,0)$ and $P_1=(-(N-2),0)$ are saddle points for the system \eqref{PSred.X0}. Indeed, a simple calculation together with the Stable Manifold Theorem \cite[Section 2.7]{Pe} shows that the unique trajectory contained in the unstable manifold of $P_0$ gets out tangent to the eigenvector $(-1,N+\sigma)$, thus entering the negative corner $(-\infty,0)\times(0,\infty)$, while the unique trajectory contained in the stable manifold of $P_1$ enters this point tangent to the eigenvector
\begin{equation}\label{eigen.P1}
\overline{w}:=(1,(p+1)(N-2)-(N+\sigma)).
\end{equation}
Consider now the following curve connecting $P_0$ to the critical point $P_1$ and its normal vector
\begin{equation}\label{curve}
Z=-(N+\sigma)Y-\frac{N+\sigma}{N-2}Y^2, \quad \overline{n}:=\left(\frac{N+\sigma}{N-2}(2Y+N-2),1\right).
\end{equation}
Notice (as shown in \cite[Lemma 5.1]{IMS25}) that \eqref{curve} is a trajectory of the system \eqref{PSred.X0} for $p=p_S(\sigma)$. The flow of the system \eqref{PSred.X0} across the curve \eqref{curve} has the direction given by the sign of the scalar product between the vector field of \eqref{PSred.X0} and the normal $\overline{n}$, which gives
$$
F(Y):=\frac{(N+\sigma)(p_S(\sigma)-p)}{N-2}Y^2(Y+N-2)\leq0,
$$
provided $Y\in[-(N-2),0]$. This sign implies that the region
$$
\mathcal{R}_0:=\left\{(Y,Z)\in\real^2:-(N-2)\leq Y\leq0, \ 0\leq Z\leq-\frac{N+\sigma}{N-2}Y(Y+N-2)\right\}
$$
is positively invariant for $p>p_S(\sigma)$. We prove next that the trajectory going out of $P_0$ enters the region $\mathcal{R}_0$ for $p>p_S(\sigma)$. To this end, we proceed as in \cite[Section 2.7]{Shilnikov} and look for a quadratic local approximation of the unstable manifold of $P_0$ in the form
\begin{equation}\label{interm11}
Z=-(N+\sigma)Y+KY^2+o(Y^2).
\end{equation}
We compute the expression of the flow of the system \eqref{PSred.X0} across the curve \eqref{interm11} by calculating the scalar product between the vector field of the system with the normal vector $(N+\sigma-2KY,1)$ and identify $K$ such that the expression of this scalar product be of order $o(Y^2)$. By performing straightforward calculations and then equating to zero the quadratic terms of the resulting expression to find $K$, we conclude that the curve in \eqref{interm11} reads
\begin{equation}\label{curve2}
Z=-(N+\sigma)Y-\frac{(N+\sigma)p}{N+2\sigma+2}Y^2+o(Y^2).
\end{equation}
By comparing the quadratic terms in \eqref{curve} and \eqref{curve2} and observing that
$$
\frac{(N+\sigma)p}{N+2\sigma+2}-\frac{N+\sigma}{N-2}=\frac{N+\sigma}{N+2\sigma+2}(p-p_S(\sigma))>0,
$$
we readily conclude that the trajectory \eqref{curve2} goes out of $P_0$ towards the interior of the region $\mathcal{R}_0$. It then follows that the trajectory contained in the unstable manifold of $P_0$ remains forever in the compact region $\overline{\mathcal{R}}_0$.

We prove next that the system \eqref{PSred.X0} does not have limit cycles inside the region $\mathcal{R}_0$. To this end, we employ the Dulac's Criterion \cite[Theorem 2, Section 3.9]{Pe}. We pick $a\in\real$ to be chosen later and compute the divergence
\begin{equation*}
\begin{split}
D(Y,Z):&=\frac{d}{dY}(Z^aY')+\frac{d}{dZ}(Z^aZ')=\frac{d}{dY}(-Y^2Z^a-(N-2)YZ^a-Z^{a+1})\\
&+\frac{d}{dZ}((\sigma+2)Z^{a+1}+(p-1)Z^{a+1}Y)\\
&=[(a+1)(p-1)-2]YZ^{a}+[(a+1)(\sigma+2)-(N-2)]Z^a.
\end{split}
\end{equation*}
Choosing $a=(3-p)/(1-p)$, we obtain
$$
D(Y,Z)=\left[\frac{2(\sigma+2)}{p-1}-(N-2)\right]Z^a=\frac{N-2}{p-1}(p_S(\sigma)-p)Z^a<0,
$$
since $p>p_S(\sigma)$. This proves the non-existence of limit cycles in the region $\mathcal{R}_0$.

We come back next to the stable manifold of $P_1$. It is easy to check that, for $p>p_S(\sigma)$, the eigenvector $\overline{w}$ defined in \eqref{eigen.P1} points into the complementary region to $\mathcal{R}_0$, and thus the whole trajectory contained in the stable manifold remains in the region $\real\times(0,\infty)\setminus\overline{\mathcal{R}}_0$. The stability of $P_2$, the fact that $P_2$ lies inside $\mathcal{R}_0$, the non-existence of limit cycles inside $\mathcal{R}_0$, the previous result related to the stable manifold of $P_1$ and an application of the Poincar\'e-Bendixon's Theorem (see for example \cite[Section 3.7]{Pe}) imply that the trajectory contained in the unstable manifold of $P_0$ and the plane $\{X=0\}$ enters the critical point $P_2$. 

We are now left to analyze the trajectory contained in the stable manifold of $P_1$. The flow on the vertical line $Y=-(N-2)$ (with negative sign) shows that the trajectory entering $P_1$ on its stable manifold comes from the region $\{Y>-(N-2)\}$. Observe then that the second equation in \eqref{PSred.X0} entails that $Z$ decreases while $Y<-(\sigma+2)/(p-1)$ and increases when $Y>-(\sigma+2)/(p-1)$. Assume for contradiction that there is $y_0\in[-(N-2),-(\sigma+2)/(p-1)]$ such that the trajectory entering $P_1$ presents a vertical asymptote as $Y\to y_0$. The monotonicity of the $Z$ variable and the inverse function theorem imply that the curve can thus be written as a graph of a function $Y(Z)$, with derivative
$$
Y'(Z)=-\frac{(N-2)Y(Z)+Z+Y^2(Z)}{Z[\sigma+2+(p-1)Y(Z)]}, \quad Z\in(0,\infty), \quad Y(Z)\in(-(N-2),y_0).
$$
Letting $Z\to\infty$ and $Y\to y_0$, $Y<y_0$, we find that
$$
\lim\limits_{Z\to\infty}Y'(Z)=-\frac{1}{\sigma+2+(p-1)y_0}\in\left[\frac{1}{(p-1)(N-2)-(\sigma+2)},\infty\right],
$$
which contradicts the existence of a vertical asymptote in $Z$ as $Y\to y_0$, since in this case the previous limit must be equal to zero. The latter implies that the trajectory contained in the stable manifold of $P_0$ crosses the vertical line $Y=-(\sigma+2)/(p-1)$ and thus arrives from the region $\{Y>-(\sigma+2)/(p-1)\}$, in which $Z$ is increasing. The non-existence of either limit cycles (due to the monotonicity of $Z$) and of other critical points than $Q_2$ (once we know that the unstable manifold of $P_0$ points into the interior of $\mathcal{R}_0$, while our trajectory remains outside $\mathcal{R}_0$) readily imply, by an argument of contradiction, that this trajectory comes from the critical point at infinity $Q_2$, completing the proof.
\end{proof}

We plot in Figure \ref{fig1} a visual example of the phase plane associated to the system \eqref{PSred.X0}.

\begin{figure}[ht!]
  \begin{center}
  \includegraphics[width=8.5cm,height=6cm]{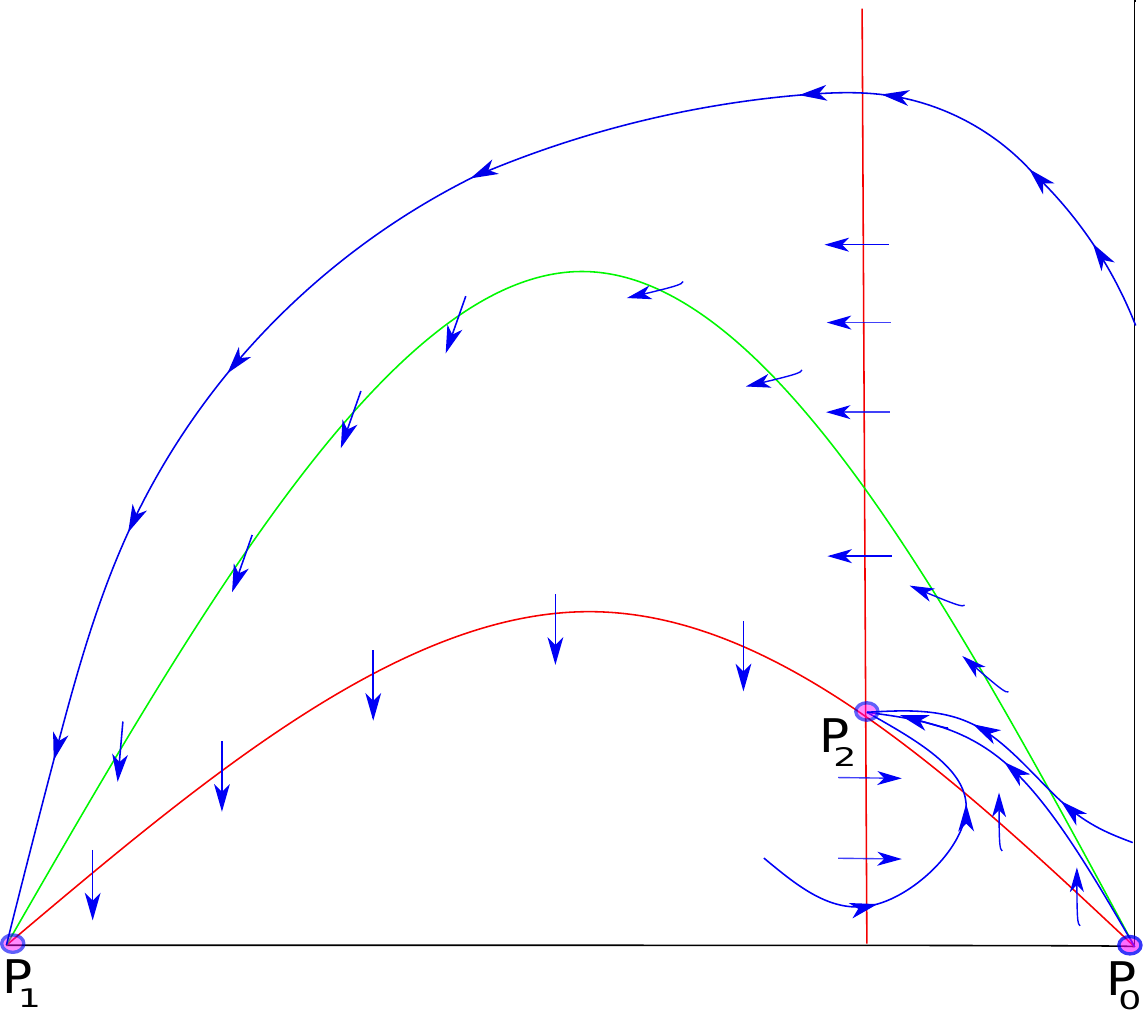}
  \end{center}
  \caption{Trajectories in the invariant plane $\{X=0\}.$}\label{fig1}
\end{figure}

We next study the unique trajectory contained in the unstable manifold of $P_0$ and in the invariant plane $\{Z=0\}$. Noticing that the system \eqref{PSsyst} reduces in the invariant plane $\{Z=0\}$ to the system
\begin{equation}\label{PSred.Z0}
\left\{\begin{array}{ll}X'=2X, \\ Y'=X-(N-2)Y-Y^2+\frac{p-1}{\sigma+2}XY,\end{array}\right.
\end{equation}
we can state the following result:
\begin{lemma}\label{lem.Z0}
The critical point $P_0=(0,0)$ in the system \eqref{PSred.Z0} is a saddle point and the unique trajectory contained in its unstable manifold (denoted by $l_0$) connects to the critical point at infinity $Q_5$. The unique trajectory entering $Q_1$ in the plane $\{Z=0\}$ on its center manifold comes from the critical point $P_1$.
\end{lemma}
\begin{proof}
It is obvious that $P_0$ is a saddle point for the system \eqref{PSred.Z0} and that the unique trajectory contained in its unstable manifold comes out tangent to the eigenvector $(N,1)$, thus it enters the positive cone $(0,\infty)^2$ of the phase plane associated to the system \eqref{PSred.Z0}. The direction of the flow of the system \eqref{PSred.Z0} across the axis $Y=0$ is given by the sign of $X$, that is, points into the positive direction, which implies that $(0,\infty)^2$ is positively invariant. Since $X'=2X>0$, it follows that $X(\eta)\to\infty$ as $\eta\to\infty$. Assume next for contradiction that $Y(\eta)/X(\eta)$ is unbounded along the trajectory, that is, there is a sequence $\{\eta_j\}_{j\geq1}$ such that
$$
\lim\limits_{j\to\infty}\frac{Y(\eta_j)}{X(\eta_j)}=\infty.
$$
The latter limit implies that the trajectory enters any neighborhood of the critical point $Q_2$, which leads to a contradiction since $Q_2$ is an unstable node. Thus, $Y/X$ is bounded along the trajectory. Employing next the change of variable \eqref{change2}, the trajectory under consideration is then seen in the system \eqref{PSinf1} as a trajectory such that $x(\eta_1)\to0$ in a decreasing way, while $y(\eta_1)$ remains bounded as $\eta_1\to\infty$. Assume next for contradiction that $y(\eta_1)$ has no limit as $\eta_1\to\infty$; that is, there are sequences $\{\eta_{1,j}^{m}\}_{j\geq1}$, respectively $\{\eta_{1,j}^{M}\}_{j\geq1}$ such that
$$
\lim\limits_{j\to\infty}\eta_{1,j}^m=\lim\limits_{j\to\infty}\eta_{1,j}^M=\infty, \quad y'(\eta_{1,j}^m)=y'(\eta_{1,j}^M)=0, \quad \eta_{1,j}^m<\eta_{1,j}^M<\eta_{1,j+1}^m
$$
and $\eta_{1,j}^m$, respectively $\eta_{1,j}^M$, are local minima, respectively local maxima, for $y(\eta_1)$ along the trajectory. Since $x(\eta_1)\to0$ as $\eta_1\to\infty$, we infer from the boundedness and positivity of $y(\eta_1)$ and the second equation in \eqref{PSinf1} (with $z=0$) that
\begin{equation}\label{interm12}
\lim\limits_{j\to\infty}y(\eta_{1,j}^m)=0, \quad \lim\limits_{j\to\infty}y(\eta_{1,j}^M)=\frac{p-1}{\sigma+2},
\end{equation}
see the proof of \cite[Proposition 4.10]{ILS24b} for a rigorous argument leading to the previous conclusion. But \eqref{interm12} produces a contradiction with the stability of the critical point $Q_5$. Indeed, as it follows from Lemma \ref{lem.Q5}, in the reduced system obtained by the restriction to the invariant plane $\{z=0\}$ of the system \eqref{PSinf1}, the point $Q_5$ has a two-dimensional center-stable manifold and thus acts as an attractor, thus attracting any trajectory reaching a neighborhood of it. We infer from this contradiction that either $y(\eta_1)\to0$ or $y(\eta_1)\to(p-1)/(\sigma+2)$ as $\eta_1\to\infty$. But the former case implies that the trajectory contained in the unstable manifold of $P_0$ enters the critical point $Q_1$ on its center manifold arriving from the positive cone $(X,Y)\in(0,\infty)^2$, a fact that is impossible in view of the local analysis of the critical point $Q_1$ performed in Lemma \ref{lem.Q1}, where the equation of the center manifold given in \eqref{cmf} entails that necessarily $y<0$ in a neighborhood of $Q_1$ on all the trajectories entering $Q_1$. We are left with $y(\eta_1)\to(p-1)/(\sigma+2)$ as $\eta_1\to\infty$, which proves that the trajectory from $P_0$ goes to $Q_5$.

Let us now shoot backwards from $Q_1$ and consider the unique trajectory contained in its (one-dimensional in the reduced system \eqref{PSred.Z0}) center manifold. It is easy to see from \eqref{cmf} and the flow on the axis $Y=0$ that the trajectory entering $Q_1$ on the center manifold arrives from the negative part $Y<0$ of the phase plane associated to the system \eqref{PSred.Z0}. The monotonicity of $X$ and a similar argument of oscillations as above show that this trajectory has to have a critical point as $\alpha$-limit, and the only possible one is $P_1$, completing the proof.
\end{proof}
For the easiness of the reading, we plot in Figure \ref{fig2} the configuration of the phase plane associated to the system \eqref{PSred.Z0}. 

\begin{figure}[ht!]
  \begin{center}
  \includegraphics[width=8.5cm,height=7cm]{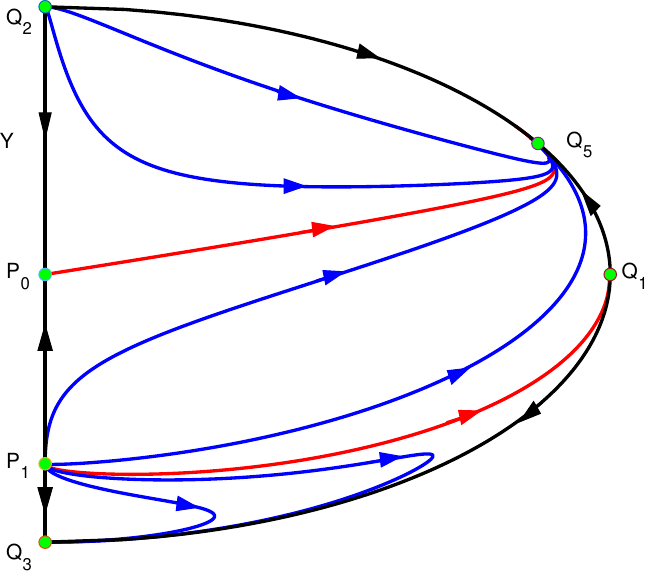}
  \end{center}
  \caption{Trajectories in the invariant plane $\{Z=0\}.$}\label{fig2}
\end{figure}

The last preparatory result is related to trajectories contained in the invariant plane $\{x=0\}$ of the system \eqref{PSinf1}. Since this system decouples, we replace the variable $z$ in \eqref{PSinf1} by $w=xz\geq0$ and only then let $x=0$ in the resulting system. We are thus left with the reduced system
\begin{equation}\label{PSsystw0}
\left\{\begin{array}{ll}\dot{y}=-y^2+\frac{p-1}{\sigma+2}y-w,\\ \dot{w}=(p-1)yw,\end{array}\right.
\end{equation}
where derivatives are taken with respect to the independent variable $\eta_1$ given in \eqref{indep2}, with finite critical points $Q_1'=(0,0)$ and $Q_5'=((p-m)/(\sigma+2),0)$, which can be seen as the ``restrictions'' of the true critical points $Q_1$ and $Q_5$ to the system \eqref{PSsystw0}.
\begin{lemma}\label{lem.w0}
The critical point $Q_5'$ is a saddle point in the system \eqref{PSsystw0} and the critical point $Q_1'$ is a saddle-node in the system \eqref{PSsystw0}. The unique orbit contained in the unstable manifold of $Q_5'$ connects to the critical point $Q_3$, with $y(\eta_1)$ decreasing for $\eta_1\in\real$. The orbits going out of $Q_1'$ connect all of them to the critical point $Q_3$ as well, with the function $y(\eta_1)$ having at most one maximum point along these orbits.
\end{lemma}
The proof is obtained by simply letting $m=1$ in the proof of \cite[Lemma 4.3]{IS25a}, where all the details are given.

\section{The main calculation. Proof of Theorems \ref{th.large} and \ref{th.small}}\label{sec.large}

This section begins with the core of the argument leading to the proof of Theorems \ref{th.large} and Theorem \ref{th.small}. The calculation that follows is similar to the argument employed by Lepin in \cite{Le89, Le90} and then also extended to quasilinear diffusion in \cite[Section 12]{GV97} in order to establish the value of the Lepin exponent \eqref{Lepin.hom}. We start from the differential equation \eqref{SSODE} and, as expected from the local analysis of the phase space performed in Section \ref{sec.syst}, we are looking for self-similar profiles $f$ with a tail given by \eqref{decay.peak} as $\xi\to\infty$. We thus perform the change of variable and function
\begin{equation}\label{var.change}
f(\xi)=\xi^{-(\sigma+2)/(p-1)}g(\xi), \quad \xi=e^{s}, \quad s\in\real.
\end{equation}
Computing first in terms of $\xi$, we find
$$
f'(\xi)=\xi^{-(\sigma+2)/(p-1)}g'(\xi)-\frac{\sigma+2}{p-1}\xi^{-(\sigma+2)/(p-1)}g(\xi)
$$
and
\begin{equation*}
\begin{split}
f''(\xi)&=\frac{\sigma+2}{p-1}\left(\frac{\sigma+2}{p-1}+1\right)\xi^{-(\sigma+2)/(p-1)-2}g(\xi)\\
&-\frac{2(\sigma+2)}{p-1}\xi^{-(\sigma+2)/(p-1)-1}g'(\xi)+\xi^{-(\sigma+2)/(p-1)}g''(\xi).
\end{split}
\end{equation*}
By inserting the previous calculations into \eqref{SSODE} and performing straightforward manipulations, we derive the differential equation solved by $g$, that is,
\begin{equation*}
\begin{split}
\xi^2g''(\xi)&-\left(\frac{2(\sigma+2)}{p-1}-N+1\right)\xi g'(\xi)+\frac{\sigma+2}{p-1}\left(\frac{\sigma+2}{p-1}+2-N\right)g(\xi)\\
&-\frac{1}{2}\xi^3g'(\xi)+g(\xi)^p=0,
\end{split}
\end{equation*}
which, in terms of the new independent variable $s=\ln\,\xi$, reads
\begin{equation}\label{ODE2}
g''(s)+Ag'(s)-\frac{B}{p-1}g(s)+g(s)^p-\frac{1}{2}e^{2s}g'(s)=0,
\end{equation}
with
\begin{equation}\label{coefs}
A=N-2-\frac{2(\sigma+2)}{p-1}, \quad B=(\sigma+2)\left(N-2-\frac{\sigma+2}{p-1}\right).
\end{equation}
We observe that the stationary solution \eqref{stat.sol} reduces to the constant solution $g\equiv C(\sigma)$, where $C(\sigma)$ is the constant defined in \eqref{stat.sol}. We thus linearize \eqref{ODE2} near the constant solution $C(\sigma)$ by setting $g(s)=C(\sigma)+y(s)$. In the previous notation, we readily deduce that the linearization of \eqref{ODE2} is given by the differential equation
\begin{equation}\label{ODE.linear}
y''(s)+Ay'(s)+By(s)-\frac{1}{2}e^{2s}y'(s)=0,
\end{equation}
which fits into the general structure of a differential equation analyzed in \cite{Le89} (with $C=1/2$ and $\gamma=2$ in the notation therein). Observe first that the condition $p>p_{JL}(\sigma)$ is equivalent to $A^2>4B$. The main conclusion of \cite{Le89} is the following: if for a fixed natural number $j\geq0$ we have
$$
4(j+1)<A-\sqrt{A^2-4B}\leq 4(j+2),
$$
then the solution to Eq. \eqref{ODE.linear} satisfying $y(s)\to1$ as $s\to\infty$ exists, is unique, and has exactly $j+2$ zeros (and the same is true for the solution to Eq. \eqref{ODE.linear} satisfying $y(s)\to K$ as $s\to\infty$ for any constant $K>0$, by linearity). We are thus interested to detect for which values of $p$ we have
\begin{equation}\label{main.ineq}
A-\sqrt{A^2-4B}>8, \quad {\rm or \ equivalently}, \quad A-8>\sqrt{A^2-4B},
\end{equation}
a fact that entails existence of solutions either by adapting the proof of \cite[Theorem 12.2]{GV97} or by an independent argument in the phase space associated to the system \eqref{PSsyst}, as we shall see below. The equivalence in \eqref{main.ineq} is justified by the fact that, for $p>p_{JL}(\sigma)$, we always have $A>8$. Indeed, the condition $A>8$ is ensured by taking
$$
p>1+\frac{2(\sigma+2)}{N-10}=:p_1
$$
and, writing $p_{JL}(\sigma)$ in the equivalent form given in \cite{MS21}
\begin{equation}\label{pJL.alt}
p_{JL}(\sigma)=1+\frac{2(\sigma+2)}{N-4-\sigma-\sqrt{(2N+\sigma-2)(\sigma+2)}},
\end{equation}
one can observe after straightforward calculations that $p_{JL}(\sigma)>p_1$ whenever $N>10+4\sigma$. Taking squares in \eqref{main.ineq} and performing easy algebraic manipulations, we obtain that \eqref{main.ineq} leads to
\begin{equation}\label{main.ineq2}
(\sigma-2)(N-2)+16>\frac{(\sigma+2)(\sigma-6)}{p-1}.
\end{equation}
We have then two cases:

$\bullet$ $\sigma\geq2$, which implies that the left-hand side of \eqref{main.ineq2} is positive and thus \eqref{main.ineq2} holds true for any
\begin{equation}\label{interm15}
p>1+\frac{(\sigma+2)(\sigma-6)}{(\sigma-2)(N-2)+16}=\frac{(\sigma-2)(N+\sigma-4)}{\sigma(N-2)-2(N-10)}.
\end{equation}
If we conventionally denote by $p_L(\sigma)$ the number in the right-hand side of \eqref{interm15}, we note that, for $\sigma\geq2$, we have
\begin{equation}\label{interm8}
p_L(\sigma)-p_S(\sigma)=-\frac{(\sigma+2)(\sigma(N-2)+2N+28)}{(N-2)[(\sigma-2)(N-2)+16]}<0,
\end{equation}
which gives that \eqref{main.ineq} holds true for any $\sigma\geq2$ and for any $p>p_S(\sigma)$.

$\bullet$ $\sigma\in(0,2)$, which implies that the sign of the left-hand side of \eqref{main.ineq2} depends on $N$. If
\begin{equation}\label{interm7}
10+4\sigma<N\leq\frac{2(10-\sigma)}{2-\sigma}:=N(\sigma),
\end{equation}
then $(\sigma-2)(N-2)+16\geq0$ and, since $(\sigma+2)(\sigma-6)<0$, it follows again that the inequality \eqref{main.ineq2} and thus also \eqref{main.ineq} hold true for any $p>1$. On the contrary, if
\begin{equation}\label{interm16}
N>\frac{2(10-\sigma)}{2-\sigma},
\end{equation}
then the left-hand side of \eqref{main.ineq2} is negative and thus similar calculations as in the first case show that \eqref{main.ineq} holds true if and only if $p<p_L(\sigma)$, where $p_L(\sigma)$ is defined in \eqref{Lepin}. Let us next prove that $p_{JL}(\sigma)<p_L(\sigma)$, whenever the condition \eqref{interm16} is fulfilled. Letting 
\begin{equation}\label{RD}
R:=\sqrt{(2N+\sigma-2)(\sigma+2)}, \quad D:=N(\sigma-2)-2(\sigma-10),
\end{equation}
we have 
$$
p_L(\sigma)-p_{JL}(\sigma)=-\frac{(\sigma+2)[(\sigma-6)R+\sigma^2+(N-6)\sigma+2N+16]}{(N-4-\sigma-R)D}.
$$
On the one hand, since $N>10+4\sigma$, it is easy to prove by taking squares that $N-4-\sigma-R>0$ (actually, this is the denominator of the exponent $p_{JL}(\sigma)$), while the condition \eqref{interm16} implies that $D<0$. On the other hand, we obtain by direct calculation that 
$$
[\sigma^2+(N-6)\sigma+2N+16]^2-[(6-\sigma)R]^2=(N\sigma+2N+6\sigma-20)^2>0,
$$
which implies that 
$$
\sigma^2+(N-6)\sigma+2N+16>(6-\sigma)R.
$$
Joining all these inequalities, we deduce that $p_L(\sigma)-p_{JL}(\sigma)>0$, as claimed.

Gathering the previous analysis, we conclude that \eqref{main.ineq} holds true if and only if: either $\sigma\geq2$ and $p>p_{JL}(\sigma)$, or $\sigma<2$ and $10+4\sigma<N\leq N(\sigma)$ and $p>p_{JL}(\sigma)$ as well, or $\sigma<2$, $N>N(\sigma)$ and $p_{JL}(\sigma)<p<p_L(\sigma)$, where we recall that $N(\sigma)$ is defined in \eqref{interm7}.

Let now $\sigma\geq 2j$ for some natural number $j\geq2$. We then compare $A-\sqrt{A^2-4B}$ with $4(j+1)$, that is, we are interested to establish for which values of $p$ we have
\begin{equation}\label{main.ineq.mult}
A-\sqrt{A^2-4B}>4(j+1), \quad {\rm or \ equivalently}, \quad A-4(j+1)>\sqrt{A^2-4B},
\end{equation}
noting that the conditions $\sigma\geq 2j$, $p>p_{JL}(\sigma)$ and $N>10+4\sigma$ readily imply that $A>4(j+1)$. Indeed, the latter condition holds true for any
$$
p>p_j:=1+\frac{2(\sigma+2)}{N-6-4j},
$$
and $p_{JL}(\sigma)>p_j$, as it easily follows from the lower bounds on $\sigma$ and $N$. We proceed similarly as in the case of \eqref{main.ineq} by taking squares in \eqref{main.ineq.mult} and perform easy algebraic manipulations to deduce that \eqref{main.ineq.mult} implies the inequality
$$
(N-2)(\sigma-2j)+4(j+1)^2\geq\frac{(\sigma+2)(\sigma-4j-2)}{p-1},
$$
which, taking into account the positivity of the left-hand side (following from the assumption $\sigma\geq2j$), is equivalent to
\begin{equation}\label{main.ineq2.mult}
p\geq p_{L,j}(\sigma):=1+\frac{(\sigma+2)(\sigma-4j-2)}{(N-2)(\sigma-2j)+4(j+1)^2}=\frac{(\sigma-2j)(N-2+\sigma-2j)}{(N-2)(\sigma-2j)+4(j+1)^2}.
\end{equation}
Observing by direct calculation that
$$
p_{L,j}(\sigma)-p_S(\sigma)=-\frac{(\sigma+2)[(N-2)(\sigma+2)+8(j+1)^2]}{(N-2)[(N-2)(\sigma-2j)+4(j+1)^2]}<0,
$$
it follows that \eqref{main.ineq2.mult} and then \eqref{main.ineq.mult} hold true for any $p>p_{JL}(\sigma)$, provided $\sigma\geq2j$.

In the rest of this section we shall prove that \eqref{main.ineq} implies existence of at least a self-similar solution, showing thus the existence part in Theorem \ref{th.large} and proving Theorem \ref{th.small}, while \eqref{main.ineq.mult} implies the existence of at least $j$ different self-similar solutions, establishing thus the multiplicity part in the statement of Theorem \ref{th.large}.
\begin{proof}[Proof of Theorems \ref{th.large} and \ref{th.small}: existence]
We adapt the end of the proof of \cite[Theorem 12.2]{GV97} by building a similar, but self-contained, argument in the phase space associated to the system \eqref{PSsyst}. We thus perform a backward shooting on the center manifold (with stable direction of the flow) of the critical point $Q_1$ and, more precisely, on the trajectories $\mathcal{C}_k$, with $k\in(0,Z_0)$, recalling the reader that $Z_0$ is defined in \eqref{Z0} and that the trajectories $\mathcal{C}_k$ are introduced in the Remark after Lemma \ref{lem.Q1}. We first observe that the number of zeros of $y(s)$ in \eqref{ODE.linear} means the number of intersections between the graph of $g(s)$ solution to \eqref{ODE2} and the constant $C(\sigma)$ or the number of intersections of the profile $f(\xi)$ solution to \eqref{SSODE} and the stationary solution \eqref{stat.sol}. We also remark that there is a one-to-one mapping between $(f,\xi)$ and $(X,Z)$ variables, according to \eqref{PSchange}. Indeed, given $(X,Z)$ as in \eqref{PSchange}, we find that
$$
f=\left[\frac{2(p-1)}{\sigma+2}X\right]^{-(\sigma+2)/2(p-1)}Z^{1/(p-1)}, \quad \xi=\sqrt{\frac{2(p-1)}{\sigma+2}X}.
$$
Thus, an intersection of a solution to \eqref{SSODE} with the profile \eqref{stat.sol} means that the same values of $(X,Z)$ are taken in the phase space of \eqref{PSsyst}, that is, $Z=Z_0$ and $X\in(0,\infty)$ arbitrary. We thus infer that the number of intersections of $g(s)$ solution to \eqref{ODE2} with the constant solution $g\equiv C(\sigma)$ means the number of crossings of the trajectory containing $g$ with the plane $\{Z=Z_0\}$, where $Z_0$ is defined in \eqref{Z0}. Let us also observe that the flow of the system \eqref{PSsyst} across the plane $\{Z=Z_0\}$ (with normal vector $(0,0,1)$) points in the upper direction if $Y>-(\sigma+2)/(p-1)$ and in the lower direction if $Y<-(\sigma+2)/(p-1)$. We thus infer from the analysis of the linearized equation \eqref{ODE.linear} that, whenever \eqref{main.ineq} is in force, the trajectories $\mathcal{C}_k$ with $k\in(Z_0-\delta,Z_0+\delta)$ (for some $\delta>0$ sufficiently small), cross the plane $\{Z=Z_0\}$ at least three times, according to the number of zeros of $y(s)$ solution to \eqref{ODE.linear}. We next introduce the following three sets:
\begin{equation}\label{sets1}
\begin{split}
&\mathcal{U}:=\{k\in(0,Z_0): \mathcal{C}_k \ {\rm crosses} \ \{Z=Z_0\} \ {\rm exactly \ twice \ and \ starts \ from} \ Q_2\},\\
&\mathcal{W}:=\{k\in(0,Z_0): \mathcal{C}_k \ {\rm crosses} \ \{Z=Z_0\} \ {\rm at \ least \ three \ times}\},\\
&\mathcal{V}:=(0,Z_0)\setminus(\mathcal{U}\cup\mathcal{W}).
\end{split}
\end{equation}
We observe that no trajectory $\mathcal{C}_k$ with $k\neq Z_0$ is tangent to the plane $\{Z=Z_0\}$. Indeed, such a tangency only may occur at a point where the flow is equal to zero, that is, at points with $Z=Z_0$ and $Y=-(\sigma+2)/(p-1)$. But such points belong to the trajectory \eqref{stat.sol.plane} which identifies with $\mathcal{C}_{Z_0}$ and, by the uniqueness of a trajectory at a non-critical point, no other orbit $\mathcal{C}_k$ might contain such points. It is then obvious that both $\mathcal{U}$ and $\mathcal{W}$ are open sets by the continuity with respect to the parameter and the stability of $Q_2$ following from Lemma \ref{lem.Q23}. On the one hand, the previous discussion implies that $\mathcal{W}$ is non-empty and in fact there is $\delta>0$ such that $(Z_0-\delta,Z_0)\subseteq\mathcal{W}$. On the other hand, it follows from Lemmas \ref{lem.X0} and \ref{lem.Z0} that the unique trajectory $\mathcal{C}_0$ entering $Q_1$ and contained in the plane $\{Z=0\}$ comes from $P_1$ and the unique trajectory entering $P_1$ on its stable manifold starts from $Q_2$. Thus, a standard argument of continuity (tubular neighborhoods of the two trajectories mentioned above) together with the behavior near a saddle \cite[Theorem 2.9]{Shilnikov} applied in a neighborhood of $P_1$ give that there is $\vartheta>0$ such that, for any $k\in(0,\vartheta)$, the trajectory $\mathcal{C}_k$ arrives from $Q_2$ and lies inside a neighborhood, first, of the trajectory $Q_2-P_1$ and then of the trajectory $\mathcal{C}_0$ connecting $P_1$ to $Q_1$. Thus, for $k\in(0,\vartheta)$, $\mathcal{C}_k$ crosses $\{Z=Z_0\}$ exactly twice (at points lying close to the crossing points of the trajectory $Q_2-P_1$ derived in Lemma \ref{lem.X0}). Thus, $(0,\vartheta)\subseteq\mathcal{U}$.

An elementary topological argument implies then that $\mathcal{V}\neq\emptyset$ and thus there is at least one element $k_1\in\mathcal{V}$. It then follows that $\mathcal{C}_{k_1}$ crosses the plane $\{Z=Z_0\}$ at most twice (otherwise $k_1\in\mathcal{W}$ and a contradiction is reached) and, if crossing it exactly twice, then its $\alpha$-limit is not the critical point $Q_2$ (otherwise $k_1\in\mathcal{U}$ and again a contradiction is reached). Moreover, Lemma \ref{lem.P1P2} proves that the trajectory $\mathcal{C}_{k_1}$ cannot start from the critical point $P_2$. It cannot also be tangent to the plane $\{Z=Z_0\}$; indeed, if we assume that there is $\eta_0\in\real$ such that $Z(\eta_0)=Z_0$ and $Z'(\eta_0)=0$, it follows from the third equation of the system \eqref{PSsyst} that $Y(\eta_0)=-(\sigma+2)/(p-1)$, thus the tangency point is of the form $(X(\eta_0),-(\sigma+2)/(p-1),Z_0)$ and belongs to the line \eqref{stat.sol.plane}. But we deduce from the non-criticality that the unique trajectory passing through points $(X,-(\sigma+2)/(p-1),Z_0)$ with $X>0$ is the explicit line \eqref{stat.sol.plane}, which differs from $\mathcal{C}_{k_1}$. We next prove that $\mathcal{C}_{k_1}$ crosses $\{Z=Z_0\}$ exactly twice. Assume first for contradiction that the orbit does not intersect the plane $\{Z=Z_0\}$, that is, it remains forever in the region $\{Z<Z_0\}$. Since $k_1<Z_0$, we infer from \eqref{cmf} by undoing \eqref{change2} that the center manifold has (with $Z\to k_1<Z_0$) the equation
$$
Y=-\frac{\sigma+2}{p-1}+\frac{\sigma+2}{(p-1)X}(Z-Z_0)+O\left(\frac{1}{X^2}\right)<-\frac{\sigma+2}{p-1},
$$
hence the trajectory $\mathcal{C}_{k_1}$ enters $Q_1$ from the half-space $\{Y<-(\sigma+2)/(p-1)\}$. The flow of the system \eqref{PSsyst} across the plane $\{Y=-(\sigma+2)/(p-1)\}$ points into the positive direction if $Z<Z_0$ and into the negative direction if $Z>Z_0$, hence it follows from the assumption that the trajectory remains forever in the region $\{Z<Z_0\}$ that it also remains in the region $\{Y<-(\sigma+2)/(p-1)\}$, where $Z$ is decreasing with respect to $\eta$. Since $X$ is increasing with respect to $\eta$, an argument of oscillation similar to the one in the proof of \cite[Proposition 4.10]{ILS24} shows that the $\alpha$-limit of the trajectory $\mathcal{C}_{k_1}$ has to be a critical point, and the only possible one is $P_1$ in that region. But the analysis of $P_1$ in Lemma \ref{lem.P1P2} leads to a contradiction. A similar contradiction is obtained if we assume that $\mathcal{C}_{k_1}$ crosses the plane $\{Z=Z_0\}$ exactly once, as there is no critical point lying above the plane $\{Z=Z_0\}$ except for $Q_4$, and Lemma \ref{lem.Q4} prevents the fact that $Q_4$ may be the $\alpha$-limit of $\mathcal{C}_{k_1}$.

We thus infer that the trajectory $\mathcal{C}_{k_1}$ crosses the plane $\{Z=Z_0\}$ exactly twice. Taking into account the flow of the system \eqref{PSsyst} across the plane $\{Y=-(\sigma+2)/(p-1)\}$, we easily infer that the $\alpha$-limit of $\mathcal{C}_{k_1}$ is contained in the region
$$
\mathcal{D}:=\left\{(X,Y,Z)\in\real^3: X\geq0, \ Y>-\frac{\sigma+2}{p-1}, \ 0\leq Z<Z_0\right\},
$$
but it is different from the critical point $Q_2$. Once again, in the region $\mathcal{D}$, the $X$ and $Z$ coordinates are increasing with respect to $\eta$, and an oscillation argument as in the proof of Lemma \ref{lem.Z0} shows that the $\alpha$-limit must be a critical point different from $Q_2$. The only choice remains thus $P_0$, which means that $\mathcal{C}_{k_1}$ connects $P_0$ to $Q_1$, completing the proof.
\end{proof}

We plot in Figures \ref{fig3} and \ref{fig4} the result of numerical experiments confirming the difference of behavior of the trajectories in the phase space and the corresponding profiles $f$, for $\sigma\in(0,2)$ and $p<p_L(\sigma)$, respectively $p>p_L(\sigma)$, where for the simplicity of the pictures, the trajectories are projected onto the $XY$-plane.

\begin{figure}[ht!]
  \begin{center}
  \subfigure[Trajectories in the phase space for $p<p_L(\sigma)$]{\includegraphics[width=7.5cm,height=6cm]{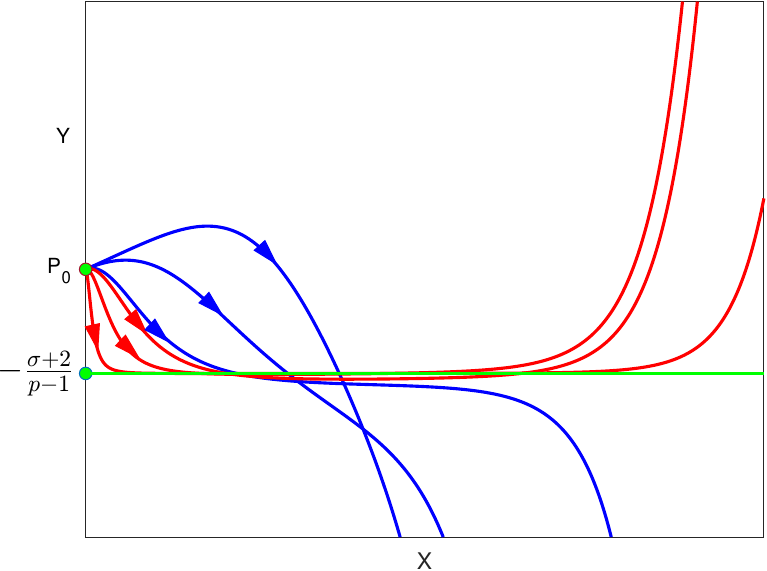}}
  \subfigure[Profiles corresponding to the trajectories for $p<p_L(\sigma)$]{\includegraphics[width=7.5cm,height=6cm]{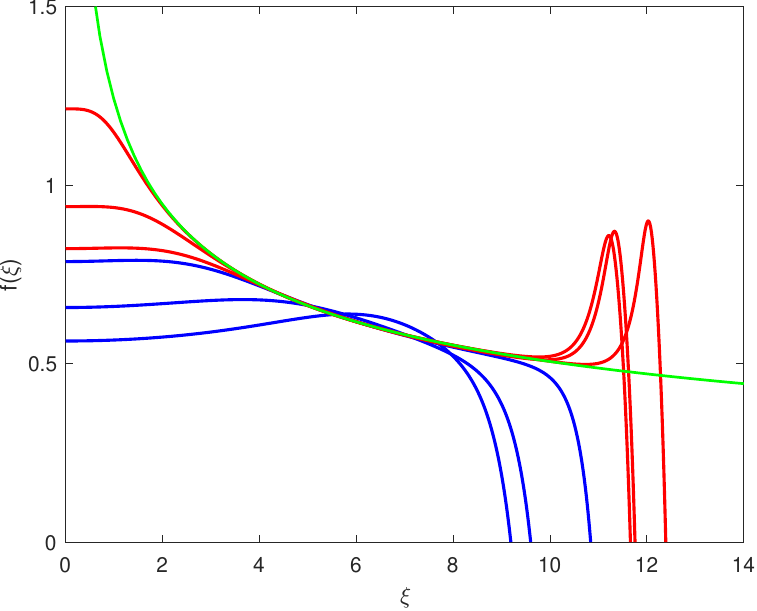}}
  \end{center}
  \caption{Trajectories and profiles for $N=20$, $\sigma=1.5$, $p=10$, where $p_{JL}(\sigma)\approx3.55$ and $p_L(\sigma)=\infty$.}\label{fig3}
\end{figure}

\begin{figure}[ht!]
  \begin{center}
  \subfigure[Trajectories in the phase space for $p>p_L(\sigma)$]{\includegraphics[width=7.5cm,height=6cm]{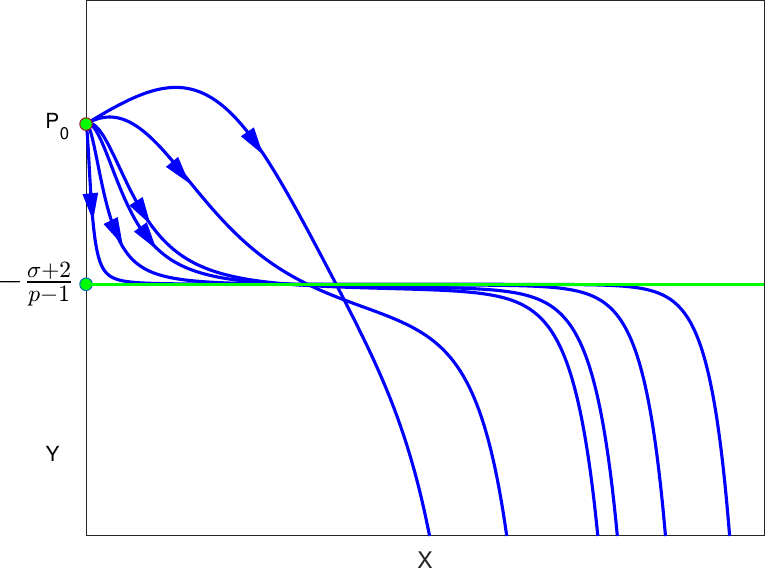}}
  \subfigure[Profiles corresponding to the trajectories for $p>p_L(\sigma)$]{\includegraphics[width=7.5cm,height=6cm]{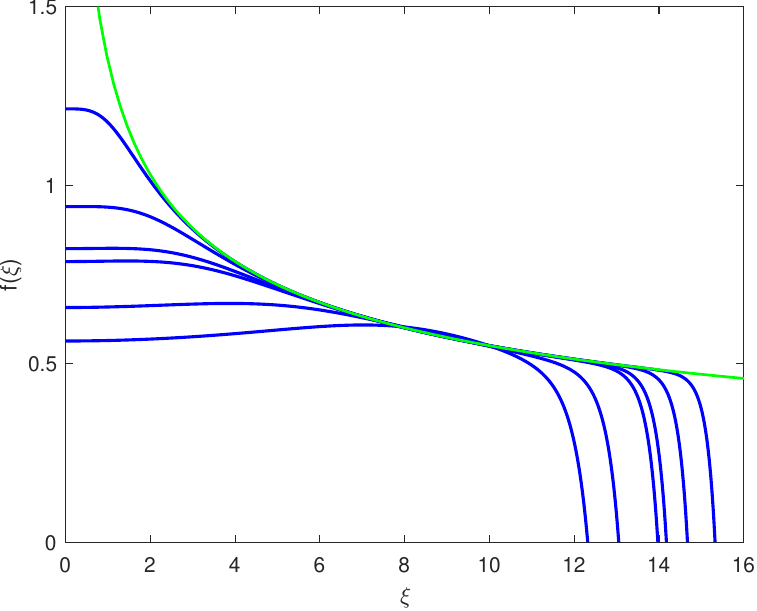}}
  \end{center}
  \caption{Trajectories and profiles for $N=40$, $\sigma=1.5$, $p=10$, where $p_{JL}(\sigma)\approx1.39$ and $p_L(\sigma)=6.25$.}\label{fig4}
\end{figure}

We observe that, in Figure \ref{fig3}, there is a clear difference of behavior of the trajectories and profiles (with, noticeably, a different number of intersections with the stationary solution \eqref{stat.sol}), while all the trajectories and profiles in Figure \ref{fig4} behave in a completely similar way, with the same number of oscillation with respect to the stationary solution \eqref{stat.sol}.

\medskip

Note that we can choose, in the previous proof, $k_1=\inf\mathcal{W}$. Based on this observation, we next prove the multiplicity statement in Theorem \ref{th.large}.
\begin{proof}[Proof of Theorem \ref{th.large}: multiplicity]
We only give here the step of passing from one self-similar solution in backward form \eqref{backwardSS} to two such solutions, the general step being then completely analogous. Assume that $\sigma\geq6$, that is, \eqref{main.ineq.mult} is in force for $j=3$ and thus implies that solutions $y(s)$ to \eqref{ODE.linear} such that $y(s)\to K\in(0,\infty)$ as $s\to\infty$ have exactly five zeros according to the analysis in \cite{Le89}. Let then $k_1=\inf\mathcal{W}\in\mathcal{V}$, where $\mathcal{W}$ and $\mathcal{V}$ are the sets introduced in \eqref{sets1}, hence it follows from the previous proof that $\mathcal{C}_{k_1}$ is a connection between $P_0$ and $Q_1$ crossing exactly twice the plane $\{Z=Z_0\}$. Let us also recall here that the stable manifold of the critical point $P_0$, according to Lemma \ref{lem.P0}, is contained in the $Y$-axis and has two trajectories, one coming from $Q_2$ along the $Y$-axis and the other one from $P_1$. By the definition of $k_1$ as the infimum of the open set $\mathcal{W}$, we infer that there is $\delta>0$ such that $k_1+\delta<Z_0$ and $(k_1,k_1+\delta)\subset\mathcal{W}$. Picking thus $k>k_1$ but sufficiently close to it, and applying a standard argument of continuity and the local behavior near the saddle $P_0$ according to \cite[Theorem 2.9]{Shilnikov}, we have the following alternative for the trajectory $\mathcal{C}_k$:

$\bullet$ either it comes from $Q_2$ following, in a neighborhood of $P_0$, the trajectory $Q_2-P_0$ belonging to the stable manifold of $P_0$. But, in this case, the $Z$ coordinate is increasing along $\mathcal{C}_k$ and thus the trajectory $\mathcal{C}_k$ crosses the plane $\{Z=Z_0\}$ exactly twice, leading to a contradiction to the fact that $k\in\mathcal{W}$. Thus, this alternative is not possible.

$\bullet$ or in a neighborhood of the saddle point $P_0$, the trajectory $\mathcal{C}_k$ arrives, according to \cite[Theorem 2.9]{Shilnikov}, from a neighborhood of the other trajectory contained in the stable manifold of $P_0$, that is, the connection $P_1-P_0$ contained in the $Y$-axis. In this case, one more application of \cite[Theorem 2.9]{Shilnikov} in a neighborhood of the saddle point $P_1$, together with the outcome of Lemma \ref{lem.X0} related to the unique orbit in the stable manifold of $P_1$ arriving from the unstable node $Q_2$, implies that the trajectory $\mathcal{C}_k$ connects $Q_2$ to $Q_1$ (following ``tubular neighborhoods" of the trajectories $Q_2-P_1$, $P_1-P_0$ and then $\mathcal{C}_{k_1}$) with exactly four crossings of the plane $\{Z=Z_0\}$. A rigorous writing of this argument with tubular neighborhoods can be seen in detail in the proof of \cite[Theorem 1.1]{IS25b}.

Let us then introduce the following sets:
\begin{equation*}
\begin{split}
&\mathcal{\widetilde{U}}:=\{k\in(k_1,Z_0): \mathcal{C}_k \ {\rm crosses} \ \{Z=Z_0\} \ {\rm exactly \ four \ times \ and \ starts \ from} \ Q_2\},\\
&\mathcal{\widetilde{W}}:=\{k\in(k_1,Z_0): \mathcal{C}_k \ {\rm crosses} \ \{Z=Z_0\} \ {\rm at \ least \ five \ times}\},\\
&\mathcal{\widetilde{V}}:=(k_1,Z_0)\setminus(\mathcal{\widetilde{U}}\cup\mathcal{\widetilde{W}}).
\end{split}
\end{equation*}
In the previous argument we have just proved that $\mathcal{\widetilde{U}}$ is non-empty and actually contains an interval of the form $(k_1,k_1+\delta)$ for some $\delta>0$. The fact that \eqref{main.ineq.mult} applies with $j=3$ and the analysis of \eqref{ODE.linear} performed in \cite{Le89} imply that there is $\delta'>0$ small such that the solution $g(s)$ to \eqref{ODE.linear} corresponding to a limit $g(s)\to L\in(C(\sigma)-\delta',C(\sigma))$ as $s\to\infty$ has at least five intersections with the constant solution $g\equiv C(\sigma)$ and thus, in terms of our phase space associated to the system \eqref{PSsyst}, the trajectory $\mathcal{C}_k$ with $k\in(Z_0-\delta',Z_0)$ crosses at least five times the plane $\{Z=Z_0\}$, that is, $(Z_0-\delta',Z_0)\subseteq\mathcal{\widetilde{U}}$. This proves the non-emptiness of $\mathcal{\widetilde{U}}$ as well. Since both $\mathcal{\widetilde{U}}$ and $\mathcal{\widetilde{W}}$ are open sets by their definition and the stability of $Q_2$, it follows that $\mathcal{\widetilde{V}}\neq\emptyset$. Picking $k_2\in\mathcal{\widetilde{V}}$, the fact that the trajectory $\mathcal{C}_{k_2}$ connects $P_0$ to $Q_1$ follows completely analogously to the last part of the previous proof, and we omit here the details.

We then observe that we can generalize the previous argument as follows: a trajectory $\mathcal{C}_{k}$ with $k$ in a suitably small neighborhood of $k_2$ will cross the plane $\{Z=Z_0\}$ at most six times (by a similar argument with tubular neighborhoods and local behavior near the saddles $P_0$ and $P_1$), while, if $\sigma\geq10$, then \eqref{main.ineq.mult} is in force for $j=5$ and ensures that trajectories $\mathcal{C}_k$ with $k$ in a suitably small left-neighborhood of $Z_0$ will cross the plane $\{Z=Z_0\}$ at least seven times. This implies, by the same three-sets argument as above, the existence of a third trajectory, $\mathcal{C}_{k_3}$, with $Z_0>k_3>k_2>k_1>0$, connecting $P_0$ to $Q_1$. And in this way, when $\sigma\geq4k-2$, we have at least $k$ different trajectories and thus $k$ different self-similar solutions with decay given by \eqref{decay.peak} as $\xi\to\infty$, as claimed.
\end{proof}

\section{Singular potentials: proof of Theorem \ref{th.neg}}\label{sec.neg}

This section is aimed at the proof of Theorem \ref{th.neg}, so that we assume throughout this section that $\sigma\in(-2,0)$. We follow ideas employed in \cite{IS25b} and, instead of the previous strategy of shooting backward (that is, in the opposite direction of the flow) from the critical point $Q_1$, we construct the proof on a shooting technique on the two-dimensional unstable manifold of $P_0$. We need a number of preparatory lemmas. The first one shows that there are no trajectories contained in the two-dimensional unstable manifold of $P_0$ that can be tangent to either one of the planes $\{Y=0\}$ (from the negative side) and $\{Y=-(\sigma+2)/(p-1)\}$.
\begin{lemma}\label{lem.notangent}
Let $(X(\eta),Y(\eta),Z(\eta))$ be a trajectory of the system \eqref{PSsyst} belonging to the unstable manifold stemming from $P_0$. Then there is no $\eta_0\in\real$ such that either $Y(\eta_0)=0$, $Y'(\eta_0)=0$ and $Y''(\eta_0)\leq0$, or $Y(\eta_0)=-(\sigma+2)/(p-1)$ and $Y'(\eta_0)=0$.
\end{lemma}
\begin{proof}
Assume for contradiction that $Y(\eta_0)=Y'(\eta_0)=0$. We infer from the second equation in \eqref{PSsyst} that $X(\eta_0)=Z(\eta_0)$. Then, by differentiating the second equation in \eqref{PSsyst}, we deduce that
$$
Y''(\eta_0)=X'(\eta_0)-Z'(\eta_0)=2X(\eta_0)-(\sigma+2)Z(\eta_0)=-\sigma X(\eta_0)>0,
$$
and a contradiction, since $X(\eta_0)=Z(\eta_0)>0$ for any $\eta_0\in\real$. Assume next by contradiction that there is $\eta_0\in\real$ such that $Y(\eta_0)=-(\sigma+2)/(p-1)$ and $Y'(\eta_0)=0$. We deduce from the second equation in \eqref{PSsyst} that $Z(\eta_0)=Z_0$, where $Z_0$ is defined in \eqref{Z0}. It thus follows that the trajectory passes through a point of the line $\{Y=-(\sigma+2)/(p-1), Z=Z_0\}$ with $X>0$. But this line is itself a trajectory, as shown in \eqref{stat.sol.plane}, and the uniqueness entails that the trajectory under consideration coincides with \eqref{stat.sol.plane}. But this is a contradiction, since the trajectory \eqref{stat.sol.plane} does not belong to the unstable manifold of $P_0$, completing the proof.
\end{proof}
The second preparatory lemma is in fact the core of the argument for the existence of (at least) a self-similar solution.
\begin{lemma}\label{lem.strip}
Let $(X,Y,Z)(\eta)$ be a trajectory of the system \eqref{PSsyst} such that there exists $\eta_*\in\real$ such that
\begin{equation}\label{cond.strip}
X(\eta_*)>0, \quad Z(\eta_*)>0, \quad -\frac{\sigma+2}{p-1}<Y(\eta)<0, \quad {\rm for \ any} \ \eta\in(\eta_*,\infty).
\end{equation}
Then the trajectory $(X,Y,Z)(\eta)$ enters the critical point $Q_1$ as $\eta\to\infty$.
\end{lemma}
\begin{proof}
The condition \eqref{cond.strip} together with the first and the third equation of the system \eqref{PSsyst} ensure that $X(\eta)$ and $Z(\eta)$ are increasing functions of $\eta$ for $\eta>\eta_*$ along the trajectory. We next show that necessarily $X(\eta)\to\infty$ as $\eta\to\infty$. Since both $X(\eta)$ and $Z(\eta)$ are monotone, there is
$$
X_{\infty}:=\lim\limits_{\eta\to\infty}X(\eta)\in(0,\infty], \quad Z_{\infty}:=\lim\limits_{\eta\to\infty}Z(\eta)\in(0,\infty].
$$
Assume for contradiction that $X_{\infty}<\infty$. Then either $Z_{\infty}<\infty$, which entails that the $\omega$-limit of the trajectory is a finite critical point with coordinates satisfying \eqref{cond.strip}, and this is impossible since such a point does not exist, or $Z_{\infty}=\infty$. In this latter case, since $Y$ is uniformly bounded for $\eta>\eta_*$, we deduce that
$$
\lim\limits_{\eta\to\infty}\frac{Z(\eta)}{Y(\eta)}=-\infty, \quad \lim\limits_{\eta\to\infty}\frac{Z(\eta)}{X(\eta)}=\infty,
$$
hence the trajectory under consideration connects to $Q_4$, contradicting Lemma \ref{lem.Q4}. It thus follows that $X_{\infty}=\infty$ and thus $Y/X\to0$ as $\eta\to\infty$. Moreover, the inverse function theorem ensures that the component $Z$ can be written as a graph of a function of $X$ with
$$
Z'(X)=\frac{Z(X)[\sigma+2+(p-1)Y]}{2X}\leq \frac{(\sigma+2)Z(X)}{2X}.
$$
By comparison with the solutions to the equation
$$
z'(X)=\frac{(\sigma+2)z(X)}{2X},
$$
starting from an arbitrary point of the trajectory as initial condition, we deduce that
$$
Z(X)\leq z(X)=KX^{(\sigma+2)/2}, \quad K>0,
$$
and thus $Z/X\to0$ as $\eta\to\infty$ along the trajectory under consideration. We have thus proved that
$$
X(\eta)\to\infty, \quad \frac{Y(\eta)}{X(\eta)}\to0, \quad \frac{Z(\eta)}{X(\eta)}\to0, \quad {\rm as} \ \eta\to\infty,
$$
which implies that the trajectory enters the critical point $Q_1$, completing the proof.
\end{proof}
We need one more preparatory result before going to the actual proof of Theorem \ref{th.neg}.
\begin{lemma}\label{lem.cylinder}
The trajectories $(l_C)_{C\in(0,\infty)}$ on the unstable manifold of $P_0$ lie above (in the sense of a larger $Z$-coordinate) the parabolic cylinder
\begin{equation}\label{cyl2}
\mathcal{H}=\left\{(X,Y,Z)\in\real^3: Z=-(N-2)Y-Y^2, \ X\geq0, \ -\frac{\sigma+2}{p-1}\leq Y\leq0\right\}
\end{equation}
\end{lemma}
\begin{proof}
On the one hand, we recall from \eqref{lC} that the trajectories contained in the unstable manifold of $P_0$ satisfy, as $\eta\to-\infty$, that
$$
Z(\eta)\sim\frac{(N+\sigma)X(\eta)}{N}-(N+\sigma)Y(\eta)\geq-(N+\sigma)Y(\eta)>-(N-2)Y(\eta),
$$
while the surface of the cylinder \eqref{cyl2} satisfies
$$
Z=-(N-2)Y-Y^2\leq-(N-2)Y.
$$
Thus, in a left-neighborhood of $P_0=(0,0,0)$, the trajectories $l_C$ with $C\in(0,\infty)$ start above the cylinder \eqref{cyl2}. On the other hand, the direction of the flow of the system \eqref{PSsyst} across the cylinder $\mathcal{H}$, with normal direction $\overline{n}=(0,N-2+2Y,1)$, is given by the sign of the scalar product between $\overline{n}$ and the vector field of the system \eqref{PSsyst}, which gives the expression
$$
E(X,Y):=\frac{\sigma+2+(p-1)Y}{\sigma+2}\left[X(N+2Y-2)-(\sigma+2)Y(N-2+Y)\right].
$$
Since $(X,Y,Z)\in\mathcal{H}$, it readily follows that $-Y(N-2+Y)\geq0$ and $\sigma+2+(p-1)Y>0$. Moreover,
$$
N+2Y-2>N-2-\frac{2(\sigma+2)}{p-1}=\frac{p(N-2)-(N+2\sigma+2)}{p-1}=\frac{N-2}{p-1}(p-p_S(\sigma))>0.
$$
We thus infer that the flow of the system \eqref{PSsyst} points outwards the cylinder \eqref{cyl2} and thus the trajectories $l_C$ cannot cross the cylinder towards its interior, completing the proof.
\end{proof}
We next recall that the analysis of the linearized equation \eqref{ODE.linear} in \cite[Theorem]{Le89} implies that the solutions $y(s)$ to \eqref{ODE.linear} such that $y(s)\to K\in(0,\infty)$ as $s\to\infty$ have exactly two zeros provided
\begin{equation}\label{interm14}
A-\sqrt{A^2-4B}>4,
\end{equation}
and a single zero in the opposite case, where we recall that $A$ and $B$ are defined in \eqref{coefs}. We first require that $A>4$, which leads to $N>6$ and
$$
p>p_0:=1+\frac{2(\sigma+2)}{N-6}.
$$
This inequality is automatically fulfilled when $p_{JL}(\sigma)\geq p_0$. Recalling the alternative form of $p_{JL}(\sigma)$ given in \eqref{pJL.alt}, the inequality $p_{JL}(\sigma)\geq p_0$ is equivalent to
$$
N-6\geq N-4-\sigma-\sqrt{(2N+\sigma-2)(\sigma+2)},
$$
which leads after simple calculations to the condition \eqref{interm13}. Assuming \eqref{interm13}, the inequality \eqref{interm14} is equivalent, after easy algebraic manipulations, to $B-2A+4>0$, which gives
\begin{equation}\label{main.ineq.neg}
\sigma(N-2)+4>\frac{(\sigma+2)(\sigma-2)}{p-1}.
\end{equation}
We have thus two cases:

$\bullet$ either
$$
\sigma(N-2)+4\geq0 \quad {\rm or \ equivalently} \quad N\leq\frac{2\sigma-4}{\sigma},
$$
when the inequality \eqref{main.ineq.neg} is then trivially satisfied for any $p>1$, since the right-hand side is negative.

$\bullet$ or $\sigma(N-2)+4<0$ and in this case the inequality \eqref{main.ineq.neg} is fulfilled for
$$
p<\frac{\sigma(N-2+\sigma)}{\sigma(N-2)+4}=\overline{p_L}(\sigma),
$$
(recalling that $\overline{p_L}(\sigma)$ is defined in \eqref{Lepin.neg}). Noticing that 
$$
\overline{p_L}(\sigma)-p_{JL}(\sigma)=-\frac{(\sigma+2)[(\sigma-2)R+\sigma^2+N(\sigma+2)-2\sigma]}{[\sigma(N-2)+4](N-4-\sigma-R)},
$$
with $R$ defined in \eqref{RD}, and that
$$
(\sigma^2+N(\sigma+2)-2\sigma)^2-((2-\sigma)R)^2=(N\sigma+2N+2\sigma-4)^2>0,
$$
one can show in a completely similar way as in the case $\sigma>0$ that $\overline{p_L}(\sigma)>p_{JL}(\sigma)$ whenever the condition $\sigma(N-2)+4<0$ is fulfilled.

Assume thus from now on that we are in one of the two previous cases in which \eqref{main.ineq.neg} holds true. We need one more preparatory lemma.
\begin{lemma}\label{lem.cross}
There is $\delta_0>0$ sufficiently small such that any trajectory entering the critical point $Q_1$ with
\begin{equation}\label{almost.line}
\lim\limits_{\eta\to\infty}Z(\eta)\in(Z_0,Z_0+\delta_0)
\end{equation}
(where the previous limit exists according to \eqref{interm5}) intersects at least once the plane $\{Y=-(\sigma+2)/(p-1)\}$.
\end{lemma}
\begin{proof}
We go back to the linearized equation \eqref{ODE.linear}. The analysis performed in \cite[Theorem]{Le89} (see also \cite[Lemma 9]{Le90}) proves that, if \eqref{main.ineq.neg} is in force, the solution $y$ to \eqref{ODE.linear} such that $y(s)\to K\in(0,\infty)$ as $s\to\infty$ has exactly two zeros, and thus the solution $g(s)$ to \eqref{ODE2} intersects twice the constant solution $g\equiv C(\sigma)$. As already discussed in the proof of Theorem \ref{th.large}, this fact means that the trajectory containing the solution $g$ (and the corresponding profile $f$ defined by \eqref{var.change}) intersects twice the plane $\{Z=Z_0\}$. Since the latter number of intersections is obtained by a linearization, it stays true for the actual trajectories lying in a suitably small neighborhood of the constant one \eqref{stat.sol.plane}, that is, those satisfying \eqref{almost.line} for some $\delta_0>0$ sufficiently small. An easy argument of monotonicity of the $Z$ coordinate implies that the trajectories satisfying \eqref{almost.line} have to change the monotonicity of $Z$ at least once in order to cross (at least) twice the plane $\{Z=Z_0\}$, and a change of monotonicity of $Z$ is equivalent to an intersection with the plane $\{Y=-(\sigma+2)/(p-m)\}$, according to the third equation of the system \eqref{PSsyst}, completing the proof.
\end{proof}
With these preparations, we are now ready to give the proof of Theorem \ref{th.neg}. The strategy, completely analogous to the one employed in the proof of \cite[Theorem 1.1]{IS25b}, is based on a construction of tubular neighborhoods in order to contradict the possibility that all the trajectories contained in the unstable manifold of $P_0$ remain in the half-space $\{Y>-(\sigma+2)/(p-1)\}$. Due to the analogy with the above mentioned reference, some of the steps in the construction will only be sketched, skipping some technical details that can be found in the proof of \cite[Theorem 1.1]{IS25b}.
\begin{proof}[Proof of Theorem \ref{th.neg}]
Let us split the trajectories $(l_C)_{C\in(0,\infty)}$ defined in \eqref{lC} into the following three sets:
\begin{equation}\label{sets.neg}
\begin{split}
&\mathcal{A}=\left\{C\in(0,\infty): {\rm there \ is} \ \eta_0\in\real, \ -\frac{\sigma+2}{p-1}<Y(\eta)<0 \ {\rm for \ any} \ \eta\in(-\infty,\eta_0),\right.\\&\left. {\rm and} \ Y(\eta_0)=-\frac{\sigma+2}{p-1}, \ Y'(\eta_0)<0\right\},\\
&\mathcal{C}=\left\{C\in(0,\infty): {\rm there \ is} \ \eta_0\in\real, \ -\frac{\sigma+2}{p-1}<Y(\eta)<0 \ {\rm for \ any} \ \eta\in(-\infty,\eta_0),\right.\\&\left. {\rm and} \ Y(\eta_0)=0, \ Y'(\eta_0)>0\right\},\\
&\mathcal{B}=(0,\infty)\setminus(\mathcal{A}\cup\mathcal{C}).
\end{split}
\end{equation}
It is obvious by definition that the sets $\mathcal{A}$ and $\mathcal{C}$ are open. As we have proved in Lemma \ref{lem.Z0} that $l_0$ (contained in the invariant plane $\{Z=0\}$) connects $P_0$ to $Q_5$ through the positive half-plane $\{Y>0\}$, an obvious argument of continuity with respect to the parameter $C$ on the unstable manifold of $P_0$ ensures that, for very small values of $C>0$, the trajectories $l_C$ will enter the half-space $\{Y>0\}$ very shortly after leaving $P_0$ and thus $\mathcal{C}$ is non-empty. More precisely, we infer that there is $C_*>0$ such that $(0,C_*)\subseteq\mathcal{C}$.

The most complicated argument in the proof is the non-emptiness of $\mathcal{A}$. We employ for this goal the following strategy: assuming for contradiction that $\mathcal{A}=\mathcal{B}=\emptyset$ (otherwise, we have a solution by Lemma \ref{lem.strip}), we build a system of tubular neighborhoods of the trajectories (in opposite direction to their flow) \eqref{stat.sol.plane} connecting $P_2$ to $Q_1$, $l_{\infty}$ connecting $P_0$ to $P_2$ and the $Y$ axis in the region $\{Y>0\}$ connecting $Q_2$ to $P_0$. We take the neighborhoods as thin as needed in order for a continuum of orbits $(l_C)_{C>C^*}$ for some $C^*>0$ very large to form a barrier for the trajectories inside these three tubular neighborhoods. This construction allows us to conclude that there is at least one trajectory connecting $Q_2$ to $Q_1$, through the three connected tubular neighborhoods, for which $Y(\eta)>-(\sigma+2)/(p-m)$ for any $\eta\in\real$, contradicting thus Lemma \ref{lem.cross}. We next give a sketch of this construction, recalling that all the details can be found by the interested reader in the proof of \cite[Theorem 1.1]{IS25b}. Note first that the assumption that $\mathcal{A}=\mathcal{B}=\emptyset$ is equivalent to say that all the trajectories $(l_C)_{C\in(0,\infty)}$ enter the strip $Y\in(-(\sigma+2)/(p-m),0)$ and remain there until crossing the plane $\{Y=0\}$ to the positive half-space $\{Y>0\}$.

By taking a neighborhood of $Q_1$ in the system \eqref{PSsystw0} and undoing the change of variable $w=xz$ and \eqref{change2}, we infer from the fact that the orbit $l_{\infty}$ enters the saddle point $P_2$ (whose unique unstable trajectory is the straight line \eqref{stat.sol.plane}) and the local behavior near a saddle point \cite[Theorem 2.9, Section 2.8]{Shilnikov} that, given $\epsilon>0$, there exists $C^*>0$ large enough (depending on $\epsilon$) such that the trajectory $l_C$ enters the neighborhood $\mathcal{V}$ defined below for any $C\in(C^*,\infty)$:
\begin{equation*}
\begin{split}
\mathcal{V}=&\left\{(X,Y,Z)\in\real^3:X>\max\left\{\frac{1}{\epsilon},\frac{\sigma+2}{(p-1)\epsilon},\sqrt{\frac{Z_0+\epsilon}{\epsilon}}\right\},\right.\\&\left. -\frac{\sigma+2}{p-1}<Y<-\frac{\sigma+2}{p-1}+\epsilon, Z_0-\epsilon<Z<Z_0+\epsilon\right\},
\end{split}
\end{equation*}
where $Z_0$ is defined in \eqref{Z0}. Complete details of the previous arguments are given in the proof of \cite[Theorem 1.1]{IS25b}. We then construct a "tubular" right-neighborhood of the straight line \eqref{stat.sol.plane} extending $\mathcal{V}$ to
\begin{equation*}
\begin{split}
\mathcal{W}=&\left\{(X,Y,Z)\in\real^3: 0\leq X\leq\max\left\{\frac{1}{\epsilon},\frac{\sigma+2}{(p-1)\epsilon},\sqrt{\frac{Z_0+\epsilon}{\epsilon}}\right\},\right.\\
&\left.-\frac{\sigma+2}{p-1}<Y<-\frac{\sigma+2}{p-1}+\epsilon, Z_0-\epsilon<Z<Z_0+\epsilon\right\}
\end{split}
\end{equation*}
We continue this tubular neighborhood with a tubular half-neighborhood of the trajectory $l_{\infty}$ (which connects $P_0$ to $P_2$, as proved in Lemma \ref{lem.X0}) defined by
\begin{equation*}
\mathcal{N}=\left\{(X,Y,Z)\in\real^3: 0\leq X<\epsilon, -(N-2)Y-Y^2<Z<Z_0+\epsilon, -\frac{\sigma+2}{p-1}\leq Y\leq0\right\}.
\end{equation*}
Moreover, Lemma \ref{lem.cylinder} ensures that the orbits $l_C$ on the unstable manifold of $P_0$ for $C>0$ sufficiently large will enter and remain in the neighborhood $\mathcal{N}$ up to the region $\mathcal{N}\cap\mathcal{W}$.

Let us shoot now backwards (in the reversed direction of the flow) from $Q_1$ with trajectories $\mathcal{C}_k$ with $k\in[Z_0,Z_0+\epsilon)$ for $\epsilon<\delta_0$ (with $\delta_0$ introduced in Lemma \ref{lem.cross}), where we recall that the definition of the trajectories $\mathcal{C}_k$ is given in Lemma \ref{lem.Q1} and the Remark after it. By continuity with respect to the parameter $k$ on the two-dimensional center manifold of $Q_1$, we deduce that, for $k-Z_0$ sufficiently small, the orbits $\mathcal{C}_k$ lie as close as we wish to the straight line \eqref{stat.sol.plane} (which coincides with $\mathcal{C}_{Z_0}$). Noticing that the flow on the plane $\{Y=-(\sigma+2)/(p-1)\}$, with normal vector $(0,1,0)$, is given by the sign of $Z_0-Z$, the trajectories $\mathcal{C}_k$ lying inside the neighborhood $\mathcal{V}\cup\mathcal{W}$ cannot cross the boundary formed by the unstable manifold of $P_0$ together with the region $\{Y=-(\sigma+2)/(p-1), Z\geq Z_0\}$ if $\epsilon$ is chosen sufficiently small (a precise and detailed choice of $\epsilon$ is completely similar to the one performed in \cite[p. 1427]{IS25b}).

We thus deduce that these trajectories remain in the neighborhood $\mathcal{V}\cup\mathcal{W}$ until approaching $P_2$, in the sense that for any such $k\in(Z_0,Z_0+\epsilon)$ there is $\eta_K\in\real$ such that $(X,Y,Z)(\eta_K)\in\mathcal{W}\cap\mathcal{N}$. By applying \cite[Theorem 2.9]{Shilnikov} in the neighborhood $\mathcal{W}\cap\mathcal{N}$ of the saddle $P_2$ we deduce that the trajectories $\mathcal{C}_k$ for $k-Z_0$ sufficiently small arrive through the neighborhood $\mathcal{N}$. Since both the unstable manifold of $P_0$ and the invariant plane $\{X=0\}$ are separatrices, it further follows that the trajectories under consideration remain in $\mathcal{N}$ until reaching the plane $\{Y=0\}$. A similar argument by applying \cite[Theorem 2.9]{Shilnikov} in a neighborhood included in $\mathcal{N}$ of the saddle point $P_0$ ensures that the trajectories $\mathcal{C}_k$ as above arrive through a "tubular" neighborhood of the unique trajectory in the stable manifold of $P_0$ in the half-space $\{Y>0\}$. But this trajectory connects the unstable node $Q_2$ to $P_0$ through the $Y$ axis, as shown in Lemma \ref{lem.P0}. The stability of $Q_2$ allows us to find trajectories $\mathcal{C}_k$ connecting $Q_2$ to $Q_1$ through the system of tubular neighborhoods constructed above. Such trajectories do not intersect the plane $\{Y=-(\sigma+2)/(p-1)\}$, in contradiction with Lemma \ref{lem.cross}. This contradiction implies that $\mathcal{A}$ is non-empty.

Since both $\mathcal{A}$ and $\mathcal{C}$ are non-empty and open, we conclude by a standard argument that also $\mathcal{B}$ is non-empty. Picking $C_0\in\mathcal{B}$, we conclude from Lemma \ref{lem.strip} that $l_{C_0}$ connects $P_0$ to $Q_1$ and thus has the desired behavior if we undo the change of variable \eqref{PSchange}, completing the proof.
\end{proof}

We plot in Figures \ref{fig5} and \ref{fig6} the outcome of numerical experiments showing how trajectories in the phase space and the corresponding profiles behave for $p<\overline{p_L}(\sigma)$, respectively $p>\overline{p_L}(\sigma)$, where for the simplicity of the pictures, the trajectories are projected onto the $XY$-plane. We observe in Figure \ref{fig5} the splitting of the trajectories between the sets $\mathcal{A}$ and $\mathcal{C}$ defined in \eqref{sets.neg}, while in Figure \ref{fig6} it is apparent that all the trajectories present a behavior corresponding to the set $\mathcal{C}$, suggesting the non-existence of any trajectory in either one of the sets $\mathcal{A}$ and $\mathcal{B}$.

\begin{figure}[ht!]
  \begin{center}
  \subfigure[Trajectories in the phase space for $p<\overline{p_L}(\sigma)$]{\includegraphics[width=7.5cm,height=6cm]{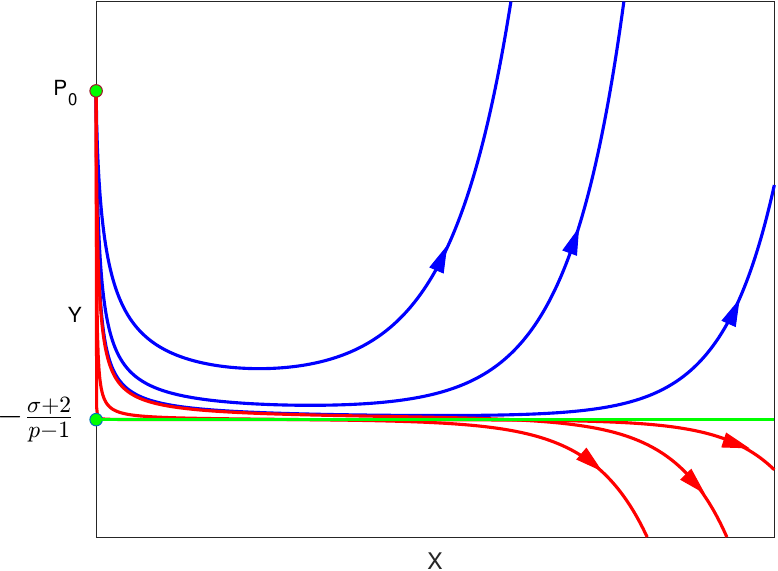}}
  \subfigure[Profiles corresponding to the trajectories for $p<\overline{p_L}(\sigma)$]{\includegraphics[width=7.5cm,height=6cm]{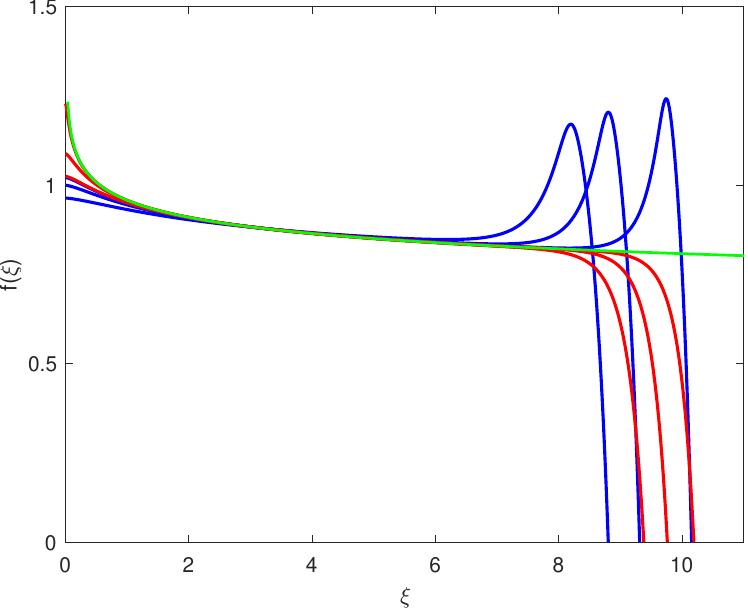}}
  \end{center}
  \caption{Trajectories and profiles for $N=8$, $\sigma=-0.6$, $p=20$, where $p_{JL}(\sigma)\approx11.4$ and $\overline{p_L}(\sigma)=\infty$.}\label{fig5}
\end{figure}

\begin{figure}[ht!]
  \begin{center}
  \subfigure[Trajectories in the phase space for $p>\overline{p_L}(\sigma)$]{\includegraphics[width=7.5cm,height=6cm]{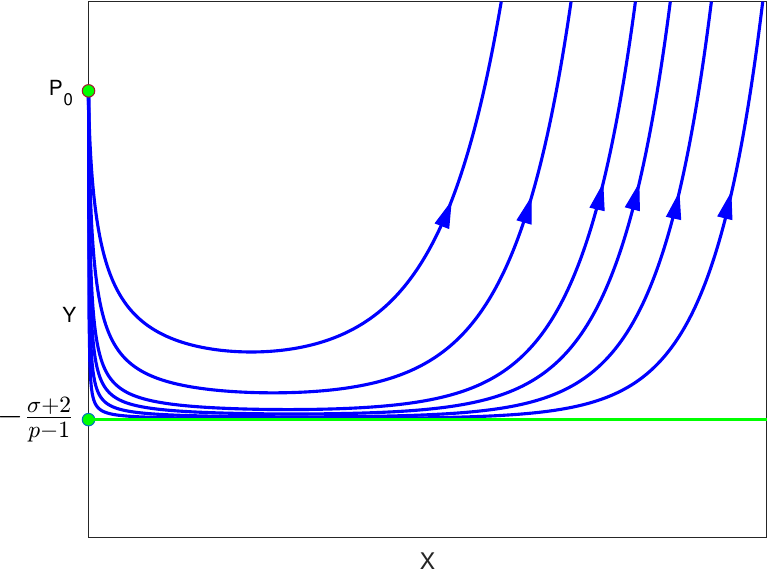}}
  \subfigure[Profiles corresponding to the trajectories for $p>\overline{p_L}(\sigma)$]{\includegraphics[width=7.5cm,height=6cm]{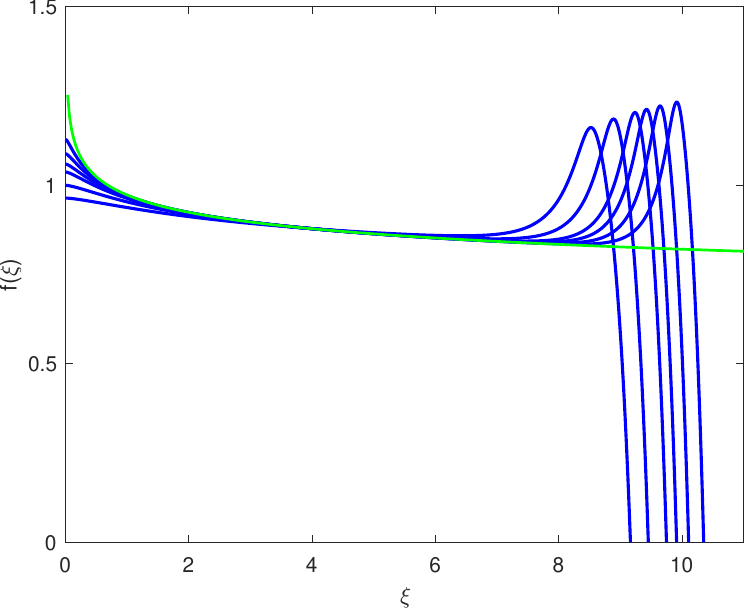}}
  \end{center}
  \caption{Trajectories and profiles for $N=10$, $\sigma=-0.6$, $p=20$, where $p_{JL}(\sigma)\approx2.68$ and $\overline{p_L}(\sigma)=5.55$.}\label{fig6}
\end{figure}

Let us finally remark that, as it is seen in Figure \ref{fig4}, respectively Figure \ref{fig6}, the reason for non-existence in the cases $\sigma\in(0,2)$ and $p>p_L(\sigma)$, respectively $\sigma\in(-2,0)$ and $p>\overline{p_L}(\sigma)$, is different. Indeed, it is apparent that in the former case, and recalling the sets defined in \eqref{sets.neg}, all the trajectories belong to the set $\mathcal{A}$, while in the latter all the trajectories belong to the set $\mathcal{C}$. This also explains the lack of continuity of $\overline{p_L}(\sigma)$ as $\sigma\to 0$ (recalling that $p_L=p_L(0)$).

\bigskip

\noindent \textbf{Acknowledgements} This work is partially supported by the Spanish project PID2024-160967NB-I00.

\bigskip

\noindent \textbf{Data availability} Our manuscript has no associated data.

\bigskip

\noindent \textbf{Conflict of interest} The authors declare that there is no conflict of interest.

\bibliographystyle{plain}

\begin{thebibliography}{1}

\bibitem{BG84}
P. Baras and J. Goldstein, \emph{The heat equation with a singular potential}, Trans. Amer. Math. Soc., \textbf{284} (1984), no. 1, 121--139.

\bibitem{BS19}
B. Ben Slimene, \emph{Asymptotically self-similar global solutions for Hardy-H\'enon parabolic systems}, Differ. Equ. Appl., \textbf{11} (2019), no. 4, 439--462.

\bibitem{BSTW17}
B. Ben Slimene, S. Tayachi and F. B. Weissler, \emph{Well-posedness, global existence and large time behavior for Hardy-H\'enon parabolic equations}, Nonlinear Anal., \textbf{152} (2017), 116--148.

\bibitem{Carr}
J. Carr, \emph{Applications of Centre Manifold Theory}, Springer Verlag, New York, 1981.

\bibitem{CIT21}
N. Chikami, M. Ikeda and K. Taniguchi, \emph{Well-posedness and global dynamics for the critical Hardy-Sobolev parabolic equation}, Nonlinearity, \textbf{34} (2021), no. 11, 8094--8142.

\bibitem{CIT22}
N. Chikami, M. Ikeda and K. Taniguchi, \emph{Optimal well-posedness and forward self-similar solution for the Hardy-H\'enon parabolic equation in critical weighted Lebesgue spaces}, Nonlinear Anal., \textbf{222} (2022), Article no.~112931, 28 pp.

\bibitem{CITT24}
N. Chikami, M. Ikeda, K. Taniguchi and S. Tayachi, \emph{Unconditional uniqueness and non-uniqueness for Hardy-H\'enon parabolic equations}, Math. Ann., \textbf{390} (2024), 3765--3825.

\bibitem{FT00}
S. Filippas and A. Tertikas, \emph{On similarity solutions of a heat equation with a nonhomogeneous nonlinearity}, J. Differential Equations, \textbf{165} (2000), no. 2, 468--492.
%
\bibitem{Fu66}
H. Fujita, \emph{On the blowing up of solutions of the Cauchy problem for $u_t=\Delta u+u^{1+\alpha}$}, J. Fac. Sci. Univ. Tokyo Sect. I, \textbf{13} (1966), 109--124.

\bibitem{GV97}
V. A. Galaktionov and J. L. V\'azquez, \emph{Continuation of blowup
solutions of nonlinear heat equations in several space dimensions},
Comm. Pure Appl. Math, \textbf{50} (1997), no. 1, 1--67.
%

\bibitem{GK85}
Y. Giga and R. Kohn, \emph{Asymptotically self-similar blow-up of semilinear heat equations}, Comm. Pure Appl. Math., \textbf{38} (1985), no. 3, 297--319.

\bibitem{GK87}
Y. Giga and R. Kohn, \emph{Characterizing blow-up using similarity variables}, Indiana Univ. Math. Journal, \textbf{36} (1987), 1--40.

\bibitem{GH}
J. Guckenheimer and Ph. Holmes, \emph{Nonlinear oscillation, dynamical systems and bifurcations of vector fields}, Applied Mathematical Sciences, vol. 42, Springer-Verlag, New York, 1990.

\bibitem{He73}
M. H\'enon, \emph{Numerical experiments on the stability of spherical stellar systems}, Astron. \& Astrophys., \textbf{24} (1973), 229--238.

\bibitem{HV94}
M. A. Herrero and J. J. L. Vel\'azquez, \emph{Explosion des solutions des \'equations paraboliques semilin\'eaires supercritiques}, C. R. Acad. Sci. Paris S\'er I Math., \textbf{319(2)} (1994), 141--145.

\bibitem{ILS24}
R. G. Iagar, M. Latorre and A. S\'anchez, \emph{Blow-up patterns for a reaction-diffusion equation with weighted reaction in general dimension}, Adv. Differential Equations, \textbf{29} (2024), no. 7-8, 515--574.

\bibitem{IL25}
R. G. Iagar and Ph. Lauren\ced{c}ot, \emph{A parabolic Hardy-H\'enon equation with quasilinear degenerate diffusion}, Submitted (2025), Preprint ArXiv no. 2503.03343.

\bibitem{IL25b}
R. G. Iagar and Ph. Lauren\ced{c}ot, \emph{A Hardy-H\'enon equation in $\mathbb{R}^N$ with sublinear absorption}, Calc. Var. Part. Differential Equations, \textbf{64} (2025), no. 3, Article no. 74, 22 p.

\bibitem{ILS24b}
R. G. Iagar, Ph. Lauren\ced{c}ot and A. S\'anchez, \emph{Self-similar shrinking of supports and non-extinction for a nonlinear diffusion equation with spatially inhomogeneous strong absorption}, Commun. Contemp. Math., \textbf{26} (2024), no. 6, Article ID 2350028, 42 p.

\bibitem{IMS23}
R. G. Iagar, A. I. Mu\~{n}oz and A. S\'anchez, \emph{Self-similar solutions preventing finite time blow-up for reaction-diffusion equations with singular potential}, J. Differential Equations, \textbf{358} (2023), 188--217.

\bibitem{IMS25}
R. G. Iagar, A. I. Mu\~{n}oz and A. S\'anchez, \emph{Extinction and non-extinction profiles for the sub-critical fast diffusion equation with weighted source}, Nonlinear Anal., \textbf{255} (2025), Paper No. 113772, 27 p.

%
\bibitem{IS22}
R. G. Iagar and A. S\'anchez, \emph{Separate variable blow-up patterns for a reaction-diffusion equation with critical weighted reaction}, Nonlinear Anal., \textbf{217} (2022), Article ID 112740, 33 p.

\bibitem{IS25a}
R. G. Iagar and A. S\'anchez, \emph{Existence and multiplicity of blow-up profiles for a quasilinear diffusion equation with source}, Qual. Theory Dyn. Syst. 24 (2025), no. 1, Paper No. 37, 55 p.

\bibitem{IS25b}
R. G. Iagar and A. S\'anchez, \emph{Existence of blow-up self-similar solutions for the supercritical quasilinear reaction-diffusion equation}, Discrete Contin. Dyn. Syst., \textbf{45} (2025), no. 5, 1399--1433.

\bibitem{JL73}
D. D. Joseph and T. S. Lundgren, \emph{Quasilinear Dirichlet problems driven by positive sources}, Arch. Rational Mech. Anal., \textbf{49} (1972/73), 241--269.

\bibitem{Le88}
L. A. Lepin, \emph{Countable spectrum of eigenfunctions of a nonlinear heat-conduction equation with distributed parameters}, Differential Equations, \textbf{24} (1988), 799--805.

\bibitem{Le89}
L. A. Lepin, \emph{The number of zeros in solutions of a second-order linear differential equation} (Russian), Latv. Gos. Univ., 1989, 37--46.

\bibitem{Le90}
L. A. Lepin, \emph{Self-similar solutions of a semilinear heat equation} (Russian), Mat. Model., \textbf{2} (1990), 63--74.

\bibitem{MM09}
H. Matano and F. Merle, \emph{Classification of type I and type II behaviors for a supercritical nonlinear heat equation}, J. Funct. Anal., \textbf{256} (2009), no. 4, 992--1064.

\bibitem{MM11}
H. Matano and F. Merle, \emph{Threshold and generic type I behaviors for a supercritical nonlinear heat equation}, J. Funct. Anal, \textbf{261} (2011), no. 3, 716--748.

\bibitem{Mi04}
N. Mizoguchi, \emph{Blowup behavior of solutions for a semilinear heat equation with supercritical nonlinearity}, J. Differential Equations, \textbf{205} (2004), 298--328.

\bibitem{Mi09}
N. Mizoguchi, \emph{Nonexistence of backward self-similar blowup solutions to a supercritical semilinear heat equation}, J. Funct. Anal., \textbf{257} (2009), 2911--2937.

\bibitem{Mi10}
N. Mizoguchi, \emph{On backward self-similar blow-up solutions to a supercritical semilinear heat equation}, Proc. Royal Soc. of Edinburgh, \textbf{140} (2010), 821--831.

\bibitem{MS21}
A. Mukai and Y. Seki, \emph{Refined construction of Type II blow-up solutions for semilinear heat equations with Joseph-Lundgren supercritical nonlinearity}, Discrete Cont. Dynamical Systems, \textbf{41} (2021), no. 10, 4847--4885.

\bibitem{Pe}
L. Perko, \emph{Differential equations and dynamical systems. Third
edition}, Texts in Applied Mathematics, \textbf{7}, Springer Verlag,
New York, 2001.

\bibitem{Qi93}
Y.-W. Qi, \emph{On the equation $u_t=\Delta u^{\alpha}+u^{\beta}$}, Proc. Roy. Soc. Edinburgh Section A, \textbf{123} (1993), no. 2, 373--390.

\bibitem{QS}
P. Quittner, and Ph. Souplet, \emph{Superlinear parabolic problems. Blow-up, global existence and steady states}, Birkhauser Advanced Texts, Birkhauser Verlag, Basel, 2007.


\bibitem{Shilnikov}
L. P. Shilnikov, A. Shilnikov, D. Turaev and L. O. Chua, \emph{Methods of qualitative theory in nonlinear dynamics. Part I}, World Scientific, 1998.
%
\bibitem{Sij}
J. Sijbrand, \emph{Properties of center manifolds}, Trans. Amer. Math. Soc., \textbf{289} (1985), no. 2, 431--469.

\bibitem{SU24}
R. Suzuki and N. Umeda, \emph{On directional blow-up for a semilinear heat equation with space-dependent reaction}, J. Funct. Anal., \textbf{287} (2024), no. 8, Article ID 110567, 33p.

\bibitem{VPME}
J. L. V\'azquez, \emph{The porous medium equation. Mathematical theory}, Oxford Monographs in Mathematics, Oxford University Press, 2007.

\end{thebibliography}

\end{document}